\documentclass{amsart}

\usepackage{amsmath,amsthm, verbatim, color}
\theoremstyle{remark} 
\usepackage{xcolor}

\theoremstyle{plain}
\newtheorem{lm}{Lemma}[section]
\newtheorem{prop}[lm]{Proposition}
\newtheorem{theorem}[lm]{Theorem}
\newtheorem{coro}[lm]{Corollary}
\theoremstyle{definition}
\newtheorem{rem}[lm]{Remark}
\newtheorem{defi}[lm]{Definition}
\newtheorem{exa}[lm]{Example}
\newtheorem*{ack}{Acknowledgments}

\usepackage[colorlinks=true,urlcolor=blue, citecolor=red,linkcolor=blue,linktocpage,pdfpagelabels, bookmarksnumbered,bookmarksopen]{hyperref}
\usepackage[hyperpageref]{backref}
\usepackage{geometry}
\usepackage{mathrsfs}

\numberwithin{equation}{section}

\title[Schr{\"o}dinger operator with confining potential]{A Schr{\"o}dinger operator\\
with confining potential having quadratic growth}



\author[Alessi]{Chiara Alessi}
\address[C.\ Alessi]{Dipartimento di Matematica e Informatica
	\newline\indent
	Universit\`a degli Studi di Ferrara
	\newline\indent
	Via Machiavelli 30, 44121 Ferrara, Italy}
\email{chiara.alessi@unife.it}

\author[Brasco]{Lorenzo Brasco}
\address[L.\ Brasco]{Dipartimento di Matematica e Informatica
	\newline\indent
	Universit\`a degli Studi di Ferrara
	\newline\indent
	Via Machiavelli 30, 44121 Ferrara, Italy}
\email{lorenzo.brasco@unife.it}

\author[Miranda]{Michele Miranda, jr.}
\address[M.\ Miranda]{Dipartimento di Matematica e Informatica
	\newline\indent
	Universit\`a degli Studi di Ferrara
	\newline\indent
	Via Machiavelli 30, 44121 Ferrara, Italy}
\email{michele.miranda@unife.it}

\begin{document}

\begin{abstract}
We study the spectral properties of a Schr\"odinger operator, in presence of a confining potential given by the distance squared from a fixed compact potential well. We prove continuity estimates on both the eigenvalues and the eigenstates, lower bounds on the ground state energy, regularity and integrability properties of eigenstates. We also get explicit decay estimates at infinity, by means of elementary nonlinear methods.
\end{abstract}

\maketitle

\begin{center}
\begin{minipage}{10cm}
\small
\tableofcontents
\end{minipage}
\end{center}

\section{Introduction}
The study reported in this paper was motivated by on-going studies of atomic Bose-Einstein condensates, namely gases of atoms cooled to nano-Kelvin temperatures where quantum mechanics dominates their motion and new states of matter occur with, e.g., superfluid behaviours~\cite{LSSY, pitaevskii2016bose}.
Of special interest in view of matter-wave optics applications are the so-called {\it atomic waveguide configurations},
where atoms are subject to an external potential of magnetic and/or optical origin that confines them in the two transverse directions, while leaving them free to move along the third, axial direction~\cite{Atomtronics}.
Steps in the theoretical description of this configurations were reported in the pioneering paper~\cite{LP}, whose physical insight triggered the need for a rigorous mathematical formulation of the problem.
\par
In specific, the general problem that one aims at considering is the
following nonlinear Schr{\"o}dinger equation
\begin{equation}
\label{EqLebPav}
i\,\partial_t \Psi(t,x)=- \Delta \Psi (t,x)+ \omega^2\,V_\Sigma(x)\,\Psi(t,x) +
a\, |\Psi(t,x)|^2\, \Psi(t,x), \qquad \text{for}\ t>0, x\in \mathbb{R}^N,
\end{equation}
with $V_\Sigma$ a confining potential which may be thought
as the square of the distance from a non-empty compact 
set $\Sigma\subseteq\mathbb{R}^N$ and $\omega\in\mathbb{R}\setminus\{0\}$, $a\in \mathbb{R}$ are given constants.
In the cited paper~\cite{LP} the set $\Sigma$ is a one-dimensional curve with a small curvature, but more configurations with sharp bends are also of experimental interest, as well as configurations where $\Sigma$ can be a two-dimensional surface~\cite{garraway2016recent}.
The general objective of our mathematical enterprise is to show how the geometry of $\Sigma$ influences
the solution $\Psi$, at least close to $\Sigma$.
\vskip.2cm\noindent
In this paper we start with the simpler case of the stationary version of \eqref{EqLebPav}, in absence of the nonlinearity (i.e. we take $a=0$). In other words, we will consider
the elliptic operator
\begin{equation}
\label{operatore}
\mathcal{H}_{\omega,\Sigma}[\Psi]:=-\Delta \Psi +\omega^2\,V_\Sigma\, \Psi,
\end{equation}
where $V_\Sigma$ has the following peculiar form
\[
V_\Sigma(x)=\Big(\mathrm{dist}(x,\Sigma)\Big)^2,\qquad \text{for every}\ x\in\mathbb{R}^N,
\]
and $\Sigma\subseteq\mathbb{R}^N$ is a compact non-empty set. In particular, $V_\Sigma$ is a confining potential, in the sense that
\begin{equation}
\label{confining}
\lim_{|x|\to+\infty} V_\Sigma(x)=+\infty.
\end{equation}
We point out
that $\mathcal{H}_{\omega,\Sigma}$ is self-adjoint and non--negative in $L^2(\mathbb{R}^N)$, 
so for its spectrum we have that $\sigma(\mathcal{H}_{\omega,\Sigma})\subseteq [0,+\infty)$.
\par
For every such a set $\Sigma$, also called {\it potential well}, we will associate the following geometric parameters
\begin{equation}
\label{deltaR}
\delta_\Sigma=\min_{x\in\Sigma} |x|\qquad \text{and}\qquad R_\Sigma=\max_{x\in\Sigma} |x|.
\end{equation}
If no confusion is possible, we shall write $V=V_\Sigma$, $\delta=\delta_\Sigma$ and $R=R_\Sigma$ for 
simplicity. 
We observe that by definition we have that $\Sigma\subseteq \overline{B_{R}(0)}$. This fact implies that
we have the estimates
\begin{equation}
\label{uppEstVSigma}
V(x)\leq |x-\overline{x}|^2, \qquad \text{for every}\ x\in \mathbb{R}^N\ \text{and}\ \overline{x}\in\Sigma,
\end{equation}
and
\begin{equation}
\label{lowEstVSigma}
V(x)\geq (|x|-R)_+^2,\qquad \text{for every}\ x\in \mathbb{R}^N.
\end{equation}
We denoted by $(\,\cdot\,)_+$ the positive part.
\par
The case $R=0$ is peculiar: indeed, in this case we have $\Sigma=\{0\}$. Accordingly, the confining potential is given by $V(x)=\omega^2\,|x|^2$ and our Schr\"odinger operator boils down to the classical {\it quantum harmonic oscillator} (see for example \cite[Example 4.2.1]{FLW} or \cite[Chapter 8, Section 3]{Tes} for the spectrum of this operator).
\begin{exa}
Apart for the case $\Sigma=\{0\}$, we will be interested 
in considering
\begin{equation}\label{SR}
\Sigma=S_R=\left\{ x=(x_1,x_2,x_3)\in \mathbb{R}^3\,:\, x_1^2+x_2^2=R^2, 
x_3=0\right\}, \qquad R>0,
\end{equation}
for which we have $\delta_\Sigma=R_\Sigma=R$,
or more generally the torus
\begin{equation}\label{TorusRr}
\Sigma=T_{r,R}=\left\{ x=(x_1,x_2,x_3)\in \mathbb{R}^3\,:\,
\left(\sqrt{x_1^2+x_2^2}-R+r\right)^2+x_3^2=r^2\right\}, 
\qquad 0<2r<R.
\end{equation}
Observe that the latter reduces to the former, when $r=0$.
\end{exa}
\begin{rem}[Reduction to the case $\omega=1$]
For simplicity, in this paper we will always normalize the physical constant $\omega\not=0$ to be $1$. Accordingly, we will write 
\[
\mathcal{H}_{\Sigma}[\Psi]:=\mathcal{H}_{1,\Sigma}[\Psi]=-\Delta \Psi +V_\Sigma\, \Psi.
\]
Actually, this is not restrictive, since one can always reduce to this case by a scale change. Indeed, observe at first that for every $t>0$ and $\Sigma\subseteq\mathbb{R}^N$ non-empty compact set, we have
\[
\mathrm{dist}(x,\Sigma)=\frac{1}{t}\,\mathrm{dist}(t\,x,t\,\Sigma),\qquad \text{for every}\ x\in\mathbb{R}^N.
\]
Thus, if for a smooth function $\Psi$ we set $\Psi_t(x)=\Psi(t\,x)$, we get
\[
\begin{split}
\mathcal{H}_{\omega,\Sigma}[\Psi_t](x)&=t^2\,\left[-\Delta \Psi(t\,x)+\frac{\omega^2}{t^2}\,\Big(\mathrm{dist}(x,\Sigma)\Big)^2\,\Psi(t\,x)\right]\\
&=t^2\,\left[-\Delta \Psi(t\,x)+\frac{\omega^2}{t^4}\,\Big(\mathrm{dist}(t\,x,t\,\Sigma)\Big)^2\,\Psi(t\,x)\right]\\
&=t^2\,\left[-\Delta \Psi(t\,x)+\frac{\omega^2}{t^4}\,V_{t\,\Sigma}(t\,x)\,\Psi(t\,x)\right]=t^2\,\mathcal{H}_{\frac{\omega}{t^2},t\,\Sigma}[\Psi](t\,x).
\end{split}
\]
By making the choice $t=\sqrt{\omega}$ for the scale parameter, we get that 
\[
\Psi\ \text{is an eigenstate of}\ \mathcal{H}_{1,\sqrt{\omega}\,\Sigma}\qquad \Longleftrightarrow\qquad \Psi_{\sqrt{\omega}}   \ \text{is an eigenstate of}\ \mathcal{H}_{\omega,\Sigma}.
\]
Accordingly, we also get the following relation between the eigenvalues
\[
\lambda_k(\omega^2\,V_\Sigma)=\omega\,\lambda_k(V_{\sqrt{\omega}\,\Sigma}).
\]
\end{rem}
\begin{rem}
We point out the recent paper \cite{Fra}, which considers a quite general class of confining potentials, comprising our class. In \cite{Fra}, a sharp lower bound on the ground state energy $\lambda_1(V)$ is given, together with a stability estimate.
\end{rem}
We now wish to briefly summarize the main results of the paper:
\begin{itemize}
\item we give at first an explicit lower bound on the ground state energy, in terms of $N$ and $R$ only, see Proposition \ref{prop:poincare};
\vskip.2cm
\item we give some stability estimates for the spectrum, by proving two-sided bounds on the difference between $\lambda_k(V_1)-\lambda_k(V_2)$. Here $V_1=V_{\Sigma_1}$ and $V_2=V_{\Sigma_2}$ are two different potentials belonging to our class. We show, through an explicit estimate, that this difference can controlled by the {\it Hausdorff distance} between the potential wells $\Sigma_1$ and $\Sigma_2$ (see Proposition \ref{prop:hausdorff});
\vskip.2cm
\item we prove some integrability estimates on the eigenstates, namely: a global $L^\infty$ estimate (Proposition \ref{prop:moser1}); weighted integrability against polynomial weights with arbitrary order of growth (Proposition \ref{prop:perognik}). We also prove higher regularity of eigenstates, up to the optimal threshold (Proposition \ref{prop:classical} and Remark \ref{oss:maxreg});
\vskip.2cm
\item in Theorem \ref{teo:expoLinfty} we prove the exponential decay of the eigenstates at infinity in the uniform norm, with an explicit a priori estimate;
\vskip.2cm
\item finally, in Theorem \ref{thm:stability} and Corollary \ref{coro:stability} we discuss what happens when $R$ is close to $0$, i.e. when the potential well $\Sigma$ approaches the origin and consequently our operator should collapse on the classical quantum harmonic oscillator. We show that this is actually the case, by giving an explicit estimate of the distance between the spectra and the eigenspaces of the two relevant operators, only in terms of the dimension $N$, the order $n$ of the eigenvalue and the geometric parameter $R$. 
\end{itemize}

\subsection{Plan of the paper}
In Section \ref{sec:2} we introduce the main notation and the basic properties of our class of potentials. The basic Spectral Theory of our operator is recalled in Section \ref{sec:3}. In Section \ref{sec:4} we show that the spectrum depends continuously on the potential well $\Sigma$ (with respect to the Hausdorff distance). In other words, if $\Sigma_1$ and $\Sigma_2$ are close in the Hausdorff topology, then the spectra must be close, as well. We then start to give some regularity estimates on the eigenstates: we prove at first some integrability estimates in Section \ref{sec:5} and then discuss higher regularity in Section \ref{sec:6}.  
In Section \ref{sec:7} we prove the exponential 
decay of eigenstates, while in Section \ref{sec:8} we give a quantitative estimate on the distance of eigenspaces from those of the quantum harmonic oscillator, in terms of the geometric parameter $R$. An appendix, containing a uniform lower bound of the ground state energy in a particular situation, concludes the paper.

\begin{ack}
The authors wish to thank Paolo Baroni and Iacopo Carusotto for some useful discussions. 
C.\,A. has been financially supported by the joint Ph.D. program of the Universities of Ferrara, Modena \& Reggio Emilia and Parma. L.\,B. has been financially supported by the {\it Fondo di Ateneo per la Ricerca} {\sc FAR 2022} and the {\it Fondo per l'Incentivazione alla Ricerca Dipartimentale} {\sc FIRD 2022} of the University of Ferrara. 
\end{ack}

\section{Preliminaries}
\label{sec:2}
\subsection{Notation and basics}
We will use the symbol $\mathbb{N}^*=\mathbb{N}\setminus\{0\}$.
In what follows, for $\mathbf{h}\in\mathbb{R}^N$ and $\varphi\in L^1_{\rm loc}(\mathbb{R}^N)$, we will use the following notation
\[
\tau_\mathbf{h}\varphi(x)=\varphi(x+\mathbf{h}),\qquad \delta_\mathbf{h} \varphi(x)=\tau_\mathbf{h}\varphi(x)-\varphi(x).
\] 
For a non-negative continuous function $\mu$ defined on $\mathbb{R}^N$, we will denote 
\[
L^2(\mathbb{R}^N;\mu)=\left\{ \phi:\mathbb{R}^N\to\mathbb{R}\ \text{measurable} \,: \,  
\int_{\mathbb{R}^N} \mu\,|\phi|^2 dx<+\infty\right\}.
\]
Here measurability will always be intended with respect to the $N-$dimensional Lebesgue measure.
In the sequel, we will need the following celebrated functional inequalities. They hold for functions in $C^\infty_c(\mathbb{R}^N)$ and thus, by density, for the whole space $H^1(\mathbb{R}^N)$, as well:
\begin{itemize}
\item for $N\ge 3$ the {\it Sobolev inequality}
\begin{equation}
\label{sobolev}
\mathcal{T}_N\,\left(\int_{\mathbb{R}^N}|\phi|^{2^*}\,dx\right)^\frac{2}{2^*}\le \int_{\mathbb{R}^N}|\nabla\phi|^2\,dx;
\end{equation}
where 
\[
2^*=\frac{2\,N}{N-2}\qquad \text{and}\qquad\mathcal{T}_N=N\,(N-2)\,\pi\,\left(\frac{\Gamma(N/2)}{\Gamma(N)}\right)^\frac{2}{N},
\]
see for example \cite[Chapter 2, Section 3.5]{Maz};
\vskip.2cm
\item for $N=2$ the {\it Ladyzhenskaya inequality}
\begin{equation}
\label{lady}
\mathcal{L}_q\,\int_{\mathbb{R}^2} |\phi|^q\,dx\le \left(\int_{\mathbb{R}^2} |\nabla \phi|^2\,dx\right)^{\frac{q-2}{2}}\,\left(\int_{\mathbb{R}^2} |\phi|^2\,dx\right),
\end{equation}
for every $2<q<\infty$.
For the special case $q=4$, which is actually the original Ladyzhenskaya inequality (see \cite{Lad}), we know that we can take $\mathcal{L}_4=\pi$ (see \cite[equation (1.11)]{Le});
\vskip.2cm
\item finally, for $N=1$ the {\it Morrey inequality}
\begin{equation}
\label{morrey}
\frac{1}{2}\,\|\phi\|_{L^\infty(\mathbb{R})}^2\le \left(\int_{\mathbb{R}}|\phi'|^2\,dx\right)^\frac{1}{2}\,\left(\int_{\mathbb{R}}|\phi|^2\,dx\right)^\frac{1}{2}.
\end{equation}
This can be easily obtained from the fundamental theorem of calculus.
\end{itemize}

\subsection{Properties of the potentials}
We start with the following two technical lemmas of general character.

\begin{lm}
\label{lm:negative}
The potential $V$ is a locally Lipschitz function such that we have 
\begin{equation}
\label{lowerbound}
|x|^2\,\left(\frac{\sigma-1}{\sigma}\right)^2\le V(x),\qquad \mbox{for every}\ \sigma>1,\ x\in 
\mathbb{R}^N\setminus B_{\sigma R}(0),
\end{equation}
and 
\begin{equation}
\label{upperbound}
V(x)\le 4\,|x|^2,\qquad \mbox{ for every } x\in 
\mathbb{R}^N\setminus B_{\delta}(0),
\end{equation}
where $\delta,R$ are defined in \eqref{deltaR}.
In particular, we have
\begin{equation}
\label{potenziale}
\frac{1}{V}\in L^p(\mathbb{R}^N\setminus B_{\sigma R}(0)),\qquad \mbox{for every}\ \sigma>1\ \text{and}\ p>\frac{N}{2}. 
\end{equation}
\end{lm}
\begin{proof}
From \eqref{uppEstVSigma}, by arbitrariness of $\overline{x}\in\Sigma$ and recalling the definition of $\delta$,
for every $x\in \mathbb{R}^N\setminus B_{\delta}(0)$ we have
\[
V(x)\le (|x|+\delta)^2\le 4\,|x|^2,
\]
and the upper bound \eqref{upperbound} follows.
From \eqref{lowEstVSigma}, for every $\sigma>1$ and every $x\in \mathbb{R}^N\setminus B_{\sigma R}(0)$ we have
\[
V(x)\ge (|x|-R)^2\ge \left(|x|-\frac{|x|}{\sigma}\right)^2.
\]
This gives the lower bound, as well.
The claimed integrability of $1/V$ now easily follows from that of $1/|x|^2$.
\end{proof}

\begin{rem}
For later use, we also record another couple of pointwise estimates on $V$.
By \eqref{uppEstVSigma} and Young's inequality, we get in particular 
\[
V(x) \leq (1+\varepsilon)\,|x|^2 +\left(1+\frac{1}{\varepsilon}\right)\,|\overline{x}|^2,\qquad \text{for every}\ x\in\mathbb{R}^N \ \text{and}\ \overline{x}\in \Sigma.
\]
By minimizing with respect to $\overline{x}\in\Sigma$, this yields
\begin{equation}
\label{eqn1}
V(x) \leq (1+\varepsilon)\,|x|^2 +\left(1+\frac{1}{\varepsilon}\right)\,\delta^2,\qquad \text{for every}\ x\in\mathbb{R}^N.
\end{equation}
Moreover, for every $\varepsilon>0$
\[
|x|^2\le (1+\varepsilon)\,(|x|-R)_+^2+\left(1+\frac{1}{\varepsilon}\right)\,R^2.
\]
By using \eqref{lowEstVSigma}, we thus obtain
\begin{equation}
\label{Vepsilon}
|x|^2\le (1+\varepsilon)\,V(x)+\left(1+\frac{1}{\varepsilon}\right)\,R^2,\qquad \text{for every}\ x\in\mathbb{R}^N,\ \varepsilon>0.
\end{equation}
\end{rem}

\begin{lm}
\label{lm:intermediate}
Let $\phi\in L^2_{\rm loc}(\mathbb{R}^N)\cap L^2(\mathbb{R}^N;V)$. Then
\[
\phi\in L^q(\mathbb{R}^N),\qquad \text{for every}\ \frac{2\,N}{N+2}<q< 2.
\]
\end{lm}
\begin{proof}
For every $\sigma>1$ and every $2N/(N+2)<q<2$ we have by H\"older's inequality
\[
\int_{B_{\sigma R}(0)} |\phi|^q\,dx\le 
|B_{\sigma R}(0)|^{1-\frac{q}{2}}\,
\left(\int_{B_{\sigma R}(0)} |\phi|^2\,dx\right)^\frac{q}{2}.
\]
On the other hand, again by H\"older's inequality we get
\[
\int_{\mathbb{R}^N \setminus B_{\sigma R}(0)} |\phi|^q\,dx\le
\left( 
\int_{\mathbb{R}^N\setminus B_{\sigma R}(0)} V\,|\phi|^2\,dx
\right)^\frac{q}{2}\,\left(\int_{\mathbb{R}^N\setminus B_{\sigma R}(0)}
\frac{1}{V^\frac{q}{2-q}}\,dx\right)^\frac{2-q}{2}.
\]
We now observe that 
\[
q>\frac{2\,N}{N+2} \qquad \Longleftrightarrow \qquad \frac{q}{2-q}>\frac{N}{2}.
\]
By Lemma \ref{lm:negative}, the last integral is finite.
\end{proof}
The weighted $L^2$ space does not depend on the particular potential $V$, under our assumptions. More precisely, we have the following
\begin{lm}
\label{lm:V}
We have 
\[
L^2_{\rm loc}(\mathbb{R}^N)\cap L^2(\mathbb{R}^N;V)=L^2_{\rm loc}(\mathbb{R}^N)\cap L^2(\mathbb{R}^N;|x|^2).
\]
Moreover, we have:
\begin{itemize}
\item for every $\phi\in L^2_{\rm loc}(\mathbb{R}^N)\cap L^2(\mathbb{R}^N;V)$ and every $\sigma>1$ 
\begin{equation}
\label{V1}
\int_{\mathbb{R}^N} |x|^2\,|\phi|^2\,dx\le \sigma^2\,R^2\,\int_{B_{\sigma R}(0)} |\phi|^2\,dx+
\left(\frac{\sigma}{\sigma-1}\right)^2\,
\int_{\mathbb{R}^N\setminus B_{\sigma R}(0)}V\,|\phi|^2\,dx;
\end{equation}
\item for every $\phi\in L^2_{\rm loc}(\mathbb{R}^N)\cap L^2(\mathbb{R}^N;|x|^2)$
\[
\int_{\mathbb{R}^N} V\,|\phi|^2\,dx\le \left(\max_{B_{\delta(0)}} V\right)\,\int_{B_{\delta}(0)} |\phi|^2\,dx+
4\,
\int_{\mathbb{R}^N\setminus B_{\delta}(0)}|x|^2\,|\phi|^2\,dx.
\]
\end{itemize}
\end{lm}
\begin{proof}
The proof is straightforward, it is sufficient to decompose the integral for every $\sigma>1$ as follows
\[
\begin{split}
\int_{\mathbb{R}^N} |x|^2\,|\phi|^2\,dx&
=\int_{B_{\sigma R}(0)} |x|^2\,|\phi|^2\,dx+
\int_{\mathbb{R}^N\setminus B_{\sigma R}(0)} |x|^2\,|\phi|^2\,dx,
\\
\end{split}
\]
and then use \eqref{lowerbound}. The second estimate can be proved similarly, by using \eqref{upperbound}: we leave the details to the reader.
\end{proof}

Finally, the following simple Poincar\'e--type inequality will be useful.
\begin{prop}
\label{prop:poincare}
There exists $C_{N,R}>0$ such that for every $\phi\in C^\infty_c(\mathbb{R}^N)$ we have
\begin{equation}
\label{ImmL2HV}
\frac{1}{C_{N,R}}\,\int_{\mathbb{R}^N} |\phi|^2\, dx\leq
\int_{\mathbb{R}^N} |\nabla\phi|^2\, dx+\int_{\mathbb{R}^N} V\,
|\phi|^2\, dx.
\end{equation}
Moreover, the constant $C_{N,R}$ has the following properties
\[
\lim_{R\to 0^+} C_{N,R}\ge 1\qquad \text{and}\qquad 0<\lim_{R\to+\infty}\frac{C_{N,R}}{R^2}<+\infty.
\]
\end{prop}
\begin{proof}
For every $\sigma>1$ we get from \eqref{lowerbound}
\[
\begin{split}
\int_{\mathbb{R}^N}|\phi|^2\,dx
&=\int_{B_{\sigma R(0)}}|\phi|^2\,dx
+
\int_{\mathbb{R}^N\setminus B_{\sigma R}(0)}|\phi|^2\, dx\\
&\le \int_{B_{\sigma R}(0)} |\phi|^2\,dx+\frac{1}{(\sigma-1)^2\,R^2}\, \int_{\mathbb{R}^N} V\,|\phi|^2\, dx.
\end{split}
\]
For $N\ge 3$, we can use 
H\"older and Sobolev inequalities \eqref{sobolev} to estimate the $L^2$ norm. We obtain
\begin{align*}
\int_{\mathbb{R}^N}|\phi|^2\,dx
&\leq |B_{\sigma R}(0)|^\frac{2}{N}\, \|\phi\|^2_{L^{2^*}(\mathbb{R}^N)}+\frac{1}{(\sigma-1)^2\,R^2}\, \int_{\mathbb{R}^N} V\,|\phi|^2\, dx\\
& \leq \frac{\omega_N^\frac{2}{N}\,(\sigma\,R)^2}{\mathcal{T}_N}\, \|\nabla \phi\|^2_{L^2(\mathbb{R}^N)}
+\frac{1}{(\sigma-1)^2\,R^2}\, \int_{\mathbb{R}^N} V\,|\phi|^2\, dx.
\end{align*}
We choose 
\[
\sigma=1+\frac{1}{R}\qquad \text{so that}\qquad \sigma\,R=R+1.
\]
In particular, we get
\[
\int_{\mathbb{R}^N}|\phi|^2\,dx\le \frac{\omega_N^\frac{2}{N}\,(R+1)^2}{\mathcal{T}_N}\, \|\nabla \phi\|^2_{L^2(\mathbb{R}^N)}
+\int_{\mathbb{R}^N} V\,|\phi|^2\, dx.
\]
By defining
\[
C_{N,R}:=\max\left\{\frac{\omega_N^\frac{2}{N}\,(R+1)^2}{\mathcal{T}_N},1\right\},
\]
we get the claimed estimate. 
\par
For $N=2$, we proceed similarly, by using this time \eqref{lady} with $q=4$, i.e
\begin{align*}
\int_{\mathbb{R}^2}|\phi|^2\,dx
&\le \int_{B_{\sigma R}(0)} |\phi|^2\,dx+\frac{1}{(\sigma-1)^2\,R^2}\, \int_{\mathbb{R}^2} V\,|\phi|^2\, dx\\
&\leq |B_{\sigma R}(0)|^\frac{1}{2}\, \|\phi\|^2_{L^{4}(\mathbb{R}^2)}+\frac{1}{(\sigma-1)^2\,R^2}\, \int_{\mathbb{R}^2\setminus B_{\sigma R}(0)} V\,|\phi|^2\, dx\\
& \leq \frac{\sqrt{\pi\,(\sigma\,R)^2}}{\sqrt{\pi}}\, \|\nabla \phi\|_{L^2(\mathbb{R}^2)}\,\|\phi\|_{L^2(\mathbb{R}^2)}
+\frac{1}{(\sigma-1)^2\,R^2}\, \int_{\mathbb{R}^2} V\,|\phi|^2\, dx.
\end{align*}
We can use Young's inequality on the first term, in order to absorb the $L^2$ norm, i.e.
\[
\int_{\mathbb{R}^2}|\phi|^2\,dx\le \frac{\sigma^2\,R^2}{2}\, \|\nabla \phi\|^2_{L^2(\mathbb{R}^2)}+\frac{1}{2}\,\|\phi\|^2_{L^2(\mathbb{R}^2)}
+\frac{1}{(\sigma-1)^2\,R^2}\, \int_{\mathbb{R}^2} V\,|\phi|^2\, dx,
\]
which leads again to the claimed estimate, upon choosing $\sigma$ as above.
\par The case $N=1$ can be treated similarly, by using \eqref{morrey} this time. The details are left to the reader.
\end{proof}

\begin{rem}
In light of Lemma \ref{lm:stimarozza} and Remark \ref{rem:Rinfinito} below, the behavior of the constant $C_{N,R}$ with respect to $R$ is optimal, in general. 
\end{rem}

\section{Spectral properties}
\label{sec:3}

We introduce the following
inner product for any $\phi,\psi\in H^1(\mathbb{R}^N)\cap L^2(\mathbb{R}^N;V)$
\[
\mathcal{Q}_V[\phi,\psi] :=
\int_{\mathbb{R}^N} \langle\nabla \phi,\nabla\psi\rangle\, dx+
\int_{\mathbb{R}^N} V\,\phi\, \psi\, dx.
\]
Accordingly, we also set
\[
\qquad\|\phi\|_V=\sqrt{\mathcal{Q}_V[\phi,\phi]},\qquad \text{for}\ \phi\in H^1(\mathbb{R}^N)\cap L^2(\mathbb{R}^N;V).
\]
It is not difficult to see that the latter is actually a norm.
\begin{defi}
With the notation above, we define the normed vector space
\[
H^1(\mathbb{R}^N;V):=H^1(\mathbb{R}^N)\cap L^2(\mathbb{R}^N;V),
\]
endowed with the norm $\|\cdot\|_V$.
\end{defi}
The following preliminary result will be important.
\begin{lm}
\label{lm:hilbert}
The space $H^1(\mathbb{R}^N;V)$ is a Hilbert space, having $C^\infty_c(\mathbb{R}^N)$ as a dense subspace. Moreover, $H^1(\mathbb{R}^N;V)$ is contained in $L^2(\mathbb{R}^N)$, with continuous inclusion. Finally, we have 
\[
H^1(\mathbb{R}^N;V)=H^1(\mathbb{R}^N;|x|^2),
\]
with equivalent norms.
\end{lm}
\begin{proof}
The density of $C^\infty_c(\mathbb{R}^N)$ follows by using standard approximation techniques. 
This in particular implies that the Poincar\'e inequality of Proposition \ref{prop:poincare} holds for functions in $H^1(\mathbb{R}^N;V)$, by density. In turn, this gives the continuity of the inclusion $H^1(\mathbb{R}^N;V)\subseteq L^2(\mathbb{R}^N)$. Thus, the norm $\|\cdot\|_V$ is equivalent to 
the norm $\|\cdot\|_{H^1(\mathbb{R}^N)}+\|\cdot\|_{L^2(\mathbb{R}^N;V)}$. By using that both $H^1(\mathbb{R}^N)$ and $L^2(\mathbb{R}^N;V)$ are Hilbert spaces, we conclude that $H^1(\mathbb{R}^N;V)$ has the same property.
\par
Finally, we have 
\[
H^1(\mathbb{R}^N;V)=H^1(\mathbb{R}^N)\cap L^2(\mathbb{R}^N;V)=H^1(\mathbb{R}^N)\cap L^2(\mathbb{R}^N;|x|^2)=H^1(\mathbb{R}^N;|x|^2),
\]
thanks to Lemma \ref{lm:V}. As for the equivalence of the norms, this is a plain consequence of the estimates of Lemma \ref{lm:V}, together with the Poincar\'e inequality of Proposition \ref{prop:poincare}.
\end{proof}
Let us now consider the self-adjoint Schr{\"o}dinger operator $\mathcal{H}_V:\mathfrak{D}(\mathcal{H}_V)\subseteq L^2(\mathbb{R}^N)\to L^2(\mathbb{R}^N)$ defined by
\[
\mathcal{H}_V[\phi]:=-\Delta \phi +V\,\phi,
\]
with domain given by\footnote{The condition on the Laplacian has to be intended in distributional sense, i.e. there exists $f\in L^2(\mathbb{R}^N)$ such that
\[
\mathcal{Q}_V[\phi,\varphi]=\int_{\mathbb{R}^N}f\, \varphi\,dx,\qquad\text{for every}\ \varphi\in C^\infty_c(\mathbb{R}^N). 
\]
In Theorem \ref{thm:domain} below, we will determine explicitly the domain $\mathfrak{D}(\mathcal{H}_V)$.}
\[
\mathfrak{D}(\mathcal{H}_V)=\Big\{\phi\in H^1(\mathbb{R}^N;V)\, :\, -\Delta \phi+V\, \phi\in L^2(\mathbb{R}^N)\Big\}.
\]
Then the quadratic form
\[
\mathcal{Q}_V[\phi,\phi]=
\int_{\mathbb{R}^N} |\nabla \phi|^2\, dx + 
\int_{\mathbb{R}^N}  V\, |\phi|^2\, dx,\qquad \text{for}\ \phi\in H^1(\mathbb{R}^N;V),
\]
is {\it associated with the operator} $\mathcal{H}_V$, in the sense that
\[
\mathfrak{D}(\mathcal{H}_V)\subseteq H^1(\mathbb{R}^N;V),
\]
and
\[
\int_{\mathbb{R}^N}\mathcal{H}_V[\phi]\,\psi\,dx=\mathcal{Q}_V[\phi,\psi],\qquad \text{for every}\ \phi\in\mathfrak{D}(\mathcal{H}_V),\, \psi\in H^1(\mathbb{R}^N;V),
\]
see \cite[Chapter 10]{BS}.
\par
In order to show that our operator has a discrete spectrum, it is now sufficient to establish the compactness of the
embedding of the form domain into $L^2(\mathbb{R}^N)$, i.e.
\[
H^1(\mathbb{R}^N;V)\hookrightarrow L^2(\mathbb{R}^N).
\]
This is the content of the next result, which proves a slightly more general assertion.
\begin{theorem}
Let $q$ be an exponent such that
\[
\left\{\begin{array}{lc}
\dfrac{2\,N}{N+2}<q<2^*,& \text{if}\ N\ge 3,\\
&\\
1<q<+\infty, &\text{if}\ N=2,\\
&\\
1\le q\le+\infty, & \text{if}\ N=1.
\end{array}
\right.
\]
Then the embedding
\[
H^1(\mathbb{R}^N;V)\hookrightarrow L^q(\mathbb{R}^N),
\]
is compact.
\end{theorem}
\begin{proof}
We first prove the result for the pivotal case $q=2$. Then, we will show how the remaining cases can be deduced from this.
\vskip.2cm\noindent
{\it Case $q=2$}. We have to prove that bounded sets in $H^1(\mathbb{R}^N;V)$  are relatively 
compact sets in $L^2(\mathbb{R}^N)$. We just need to show that the unit ball in
$H^1(\mathbb{R}^N;V)$ 
\[
{\mathcal F}=\left\{\phi \in H^1(\mathbb{R}^N;V)\,:\, \|\phi\|_V\leq 1\right\},
\]
is relatively compact in $L^2(\mathbb{R}^N)$. We will appeal to the classical Riesz--Fr\'echet--Kolmogorov Theorem and proceeds as in the proof of \cite[Theorem 10.6.5]{BS}, which concerns the one-dimensional case. We thus have to verify the following three properties:
\begin{itemize}
\item the equi-boundeness in $L^2$ norm
\[
\sup_{\phi\in\mathcal{F}} \|\phi\|_{L^2(\mathbb{R}^N)}<+\infty;
\]
\item the equi-continuity in $L^2$ norm
\[
\lim_{|\mathbf{h}|\to 0}\sup_{\phi\in\mathcal{F}} \|\tau_\mathbf{h}\phi-\phi\|_{L^2(\mathbb{R}^N)}=0;
\] 
\item the uniform mass concentration, i.e. that for every $\varepsilon>0$ there exists $R_\varepsilon>0$ such that
\[
\sup_{\phi\in\mathcal{F}} \int_{\mathbb{R}^N\setminus B_{R_\varepsilon}(0)} |\phi|^2\,dx<\varepsilon.
\]
\end{itemize}
The first condition is a plain consequence of Proposition \ref{prop:poincare}, which also holds in $H^1(\mathbb{R}^N;V)$, thanks to Lemma \ref{lm:hilbert}.
The second condition follows because by definition and Lemma \ref{lm:hilbert} we get $H^1(\mathbb{R}^N;V)\subseteq H^1(\mathbb{R}^N)$, thus we 
can use the classical inequality for the translates of Sobolev functions 
\[
\| \tau_\mathbf{h} \phi -\phi\|_{L^2(\mathbb{R}^N)} \leq |\mathbf{h}|\, \| \nabla \phi\|_{L^2(\mathbb{R}^N)} 
\leq |\mathbf{h}|\, \| \phi\|_V \leq |\mathbf{h}|.
\]
Finally, for the third condition, we only have to notice that for every $\varrho\ge 2\,R$, we have from \eqref{lowerbound} with $\sigma=2$
\[
\int_{\mathbb{R}^N\setminus B_\varrho(0)} |\phi|^2dx \leq 4\,
\int_{\mathbb{R}^N\setminus B_\varrho(0)} \frac{V}{|x|^2}\, |\phi|^2dx \leq \frac{4}{\varrho^2} \,
\|\phi\|^2_V \leq \frac{4}{\varrho^2},\qquad \text{for every}\ \phi\in\mathcal{F}.
\]
Thus, for every $\varepsilon>0$, if we choose 
\[
\varrho>\max\left\{2\,R,\frac{2}{\sqrt{\varepsilon}}\right\},
\] 
we also get the third condition of the Riesz--Fr\'echet--Kolmogorov 
Theorem. This concludes the proof for the case $q=2$. 
\vskip.2cm\noindent
{\it Case $q>2$}. This is quite standard, let us give the details for the case $N\ge 3$, the other cases being similar. By interpolation in Lebesgue spaces and Sobolev inequality \eqref{sobolev}, we have
\[
\begin{split}
\|\phi\|_{L^q(\mathbb{R}^N)}&\le \|\phi\|^\theta_{L^{2^*}(\mathbb{R}^N)}\,\|\phi\|^{1-\theta}_{L^2(\mathbb{R}^N)}\\
&\le \left(\frac{1}{\sqrt{\mathcal{T}_N}}\right)^{\theta}\,\|\nabla \phi\|^\theta_{L^2(\mathbb{R}^N)}\,\|\phi\|^{1-\theta}_{L^2(\mathbb{R}^N)},\qquad \text{for every}\ \phi\in H^1(\mathbb{R}^N).
\end{split}
\]
Here $\theta=\theta(N,q)\in(0,1)$ is the exponent dictated by scale invariance.
This shows that any sequence $\{\phi_n\}_{n\in\mathbb{N}}$ converging strongly in $L^2(\mathbb{R}^N)$ and having bounded weak gradients in $L^2(\mathbb{R}^N)$, strongly converges in $L^q(\mathbb{R}^N)$, as well. In light of the first part of the proof, this is enough to conclude.
\vskip.2cm\noindent
{\it Case $q<2$}. Let $\{\phi_n\}_{n\in\mathbb{N}}\subseteq H^1(\mathbb{R}^N;V)$ be a bounded sequence, i.e.
\begin{equation}
\label{boundq1}
\left(\|\phi_n\|^2_{L^2(\mathbb{R}^N;V)}+\|\nabla \phi_n\|^2_{L^2(\mathbb{R}^N)}\right)^\frac{1}{2}\le C,\qquad \text{for every}\ n\in\mathbb{N}.
\end{equation}
From the first part of the proof, we know that it strongly converges in $L^2(\mathbb{R}^N)$, up to a subsequence. Let us call $\phi\in L^2(\mathbb{R}^N)$ this limit function. By the lower semicontinuity of the $L^2$ norm with respect to the weak convergence, we still have $\phi \in H^1(\mathbb{R}^N;V)$, with the bound \eqref{boundq1}. In particular, by Lemma \ref{lm:intermediate} we have $\phi \in L^q(\mathbb{R}^N)$, with $q<2$ as in the statement.
We then proceed as in the proof of Lemma \ref{lm:intermediate}: we get for every $\sigma>1$ and every $n\in\mathbb{N}$
\[
\begin{split}
\int_{\mathbb{R}^N}|\phi_n-\phi|^q\,dx&= \int_{B_{\sigma R}(0)}|\phi_n-\phi|^q\,dx+\int_{\mathbb{R}^N\setminus B_{\sigma R}(0)}|\phi_n-\phi|^q\,dx\\
&\le |B_{\sigma R}(0)|^{1-\frac{q}{2}}\,
\left(\int_{B_{\sigma R}(0)} |\phi_n-\phi|^2\,dx\right)^\frac{q}{2}\\
&+\left( 
\int_{\mathbb{R}^N\setminus B_{\sigma R}(0)} V\,|\phi_n-\phi|^2\,dx
\right)^\frac{q}{2}\,\left(\int_{\mathbb{R}^N\setminus B_{\sigma R}(0)}
\frac{1}{V^\frac{q}{2-q}}\,dx\right)^\frac{2-q}{2}.
\end{split}
\]
In particular, by using \eqref{boundq1} and the strong convergence in $L^2(\mathbb{R}^N)$, we get for every $\sigma>1$
\[
\limsup_{n\to\infty}\int_{\mathbb{R}^N}|\phi_n-\phi|^q\,dx\le (2\,C)^q\,\left(\int_{\mathbb{R}^N\setminus B_{\sigma R}(0)}
\frac{1}{V^\frac{q}{2-q}}\,dx\right)^\frac{2-q}{2}.
\]
By letting $\sigma$ go to $+\infty$ and using Lemma \ref{lm:negative}, we get 
\[
\lim_{n\to\infty}\int_{\mathbb{R}^N}|\phi_n-\phi|^q\,dx=0,
\] 
thanks to the fact that $q/(2-q)>N/2$. This concludes the proof.
\end{proof}
In what follows, we will use the following notation
\[
\mathcal{S}_2(\mathbb{R}^N)=\{\phi\in L^2(\mathbb{R}^N)\, :\, \|\phi\|_{L^2(\mathbb{R}^N)}=1\}.
\]
By classical results in Spectral Theory (see for example \cite[Theorem 10.1.5]{BS}), we have the following\footnote{As customary, we repeat each eigenvalue according to its multiplicity.}
\begin{coro}
\label{coro:variational}
Under the standing assumptions, the spectrum of the operator $\mathcal{H}_V$
is made of a countable sequence of positive eigenvalues $\{\lambda_n(V)\}_{n\in\mathbb{N}^*}$ diverging to $+\infty$, each one having an associated eigenstate $\Psi_n\in H^1(\mathbb{R}^N;V)\cap \mathcal{S}_2(\mathbb{R}^N)$. The set $\{\Psi_n\}_{n\in\mathbb{N}^*}$ forms an orthonormal basis of $L^2(\mathbb{R}^N)$. Finally, each eigenvalue has the following variational characterization 
\[
\lambda_{n}(V) = \inf\left\{ \max_{\phi\in W\cap \mathcal{S}_2(\mathbb{R}^N)} \int_{\mathbb{R}^N}|\nabla \phi|^2\,dx+\int_{\mathbb{R}^N} V\,|\phi|^2\,dx\,:\, \begin{array}{c}W \subseteq H^1(\mathbb{R}^N;V)\\
\mbox{ subspace with}\ \dim W=n
\end{array} \right\},
\]
and the infimum above is attained by the vector space generated by $\{\Psi_1,\dots,\Psi_n\}$.
\end{coro}

\begin{rem}[The ground state]
\label{rem:ground}
In particular, we recall that for $n=1$ the above formulation reduces to
\[
\lambda_1(V)=\min_{\phi\in H^1(\mathbb{R}^N;V)\cap \mathcal{S}_2(\mathbb{R}^N)}\left\{\int_{\mathbb{R}^N} |\nabla \phi|^2\,dx+\int_{\mathbb{R}^N} V\,|\phi|^2\,dx\right\}.
\]
Thus, as usual $\lambda_1(V)$ coincides with the sharp constant in the Poincar\'e inequality of Proposition \ref{prop:poincare} and we have
\begin{equation}
\label{roughLB}
\lambda_1(V)\ge \frac{1}{C_{N,R}}.
\end{equation}
We also recall a couple of classical facts: any minimizer of the previous minimization problem must have constant sign, which is going to be strict by the minimum/maximum principle, and there exists a unique positive minimizer $\Psi_1$. This is also called the {\it ground state} of our operator. Thus, the first eigenvalue $\lambda_1(V)$ is {\it simple}, i.e. the corresponding eigenspace is one-dimensional. We refer for example to \cite[Theorem 11.8]{LL} for these facts.
\par
In particular, by uniqueness the ground state $\Psi_1$ must inherit all the possible symmetries
of the potential well $\Sigma$.
\end{rem}

\section{Stability of the spectrum}
\label{sec:4}

The following spectral continuity estimate will be useful. It guarantees that, at least for potential wells small enough, our operator has a spectrum quantitatively close to that of the model case $V(x)=|x|^2$.
\begin{lm}
\label{lm:stimarozza}
For every $n\in\mathbb{N}^*$ we have 
\[
\sqrt{\lambda_n}-R\le \sqrt{\lambda_n(V)}\le \sqrt{\lambda_n},
\]
where $\lambda_n$ is the $n-$th eigenvalue of the quantum harmonic oscillator.
\end{lm}
\begin{proof}
We pick a point $\overline{x}\in \Sigma$, by \eqref{uppEstVSigma}  
for every $\phi \in H^1(\mathbb{R}^N;V)$ we have that
\begin{equation}
\label{uppHVnorm}
\int_{\mathbb{R}^N} |\nabla \phi|^2\,dx
+\int_{\mathbb{R}^N} V\,|\phi|^2\, dx
\leq 
\int_{\mathbb{R}^N} |\nabla \phi|^2\,dx
+\int_{\mathbb{R}^N} |x-\overline{x}|^2\,|\phi|^2\, dx.
\end{equation}
By Corollary \ref{coro:variational}, the eigenvalues have the following variational characterizations
\[
\lambda_{n}(V) = \inf\left\{ \max_{\phi\in W\cap \mathcal{S}_2(\mathbb{R}^N)} \int_{\mathbb{R}^N}|\nabla \phi|^2\,dx+\int_{\mathbb{R}^N} V\,|\phi|^2\,dx\,:\, \begin{array}{c}W \subseteq H^1(\mathbb{R}^N;V)\\
\mbox{ subspace with}\ \dim W=n
\end{array} \right\},
\]
and
\[
\lambda_n=  \inf\left\{ \max_{\phi\in W\cap \mathcal{S}_2(\mathbb{R}^N)} \int_{\mathbb{R}^N}|\nabla \phi|^2\,dx+\int_{\mathbb{R}^N} |x|^2\,|\phi|^2\,dx\,:\, \begin{array}{c}W \subseteq H^1(\mathbb{R}^N;|x|^2)\\
\mbox{ subspace with}\ \dim W=n
\end{array} \right\},
\]
where we recall the notation
\[
\mathcal{S}_2(\mathbb{R}^N)=\{\phi\in L^2(\mathbb{R}^N)\, :\, \|\phi\|_{L^2(\mathbb{R}^N)}=1\}.
\]
Thanks to the estimate above between the quadratic forms \eqref{uppHVnorm}, to the fact that $H^1(\mathbb{R}^N;V)=H^1(\mathbb{R}^N;|x|^2)$ (see Lemma \ref{lm:hilbert}) and the translation invariance of the spectrum, we then get
\[
\lambda_n(V)\le \lambda_n.
\]
Thus, we verified the upper bound.
\par
In order to get the lower bound, we proceed similarly: it is sufficient to use \eqref{Vepsilon}.
This gives an estimate between the relevant quadratic forms: more precisely, for every $\phi\in H^1(\mathbb{R}^N;|x|^2)=H^1(\mathbb{R}^N;V)$ we get
\[
\begin{split}
\int_{\mathbb{R}^N} |\nabla \phi|^2\,dx
+\int_{\mathbb{R}^N} |x|^2\,|\phi|^2\, dx
&\leq 
(1+\varepsilon)\,\left[\int_{\mathbb{R}^N} |\nabla \phi|^2\,dx
+\int_{\mathbb{R}^N} V\,|\phi|^2\, dx\right]\\
&+\left(1+\frac{1}{\varepsilon}\right)\,R^2\,\int_{\mathbb{R}^N} |\phi|^2\,dx.
\end{split}
\]
As before, this entails that
\[
\lambda_n\le (1+\varepsilon)\,\lambda_n(V)+\left(1+\frac{1}{\varepsilon}\right)\,R^2.
\]
This is valid for every $\varepsilon>0$: upon choosing $\varepsilon=R/\sqrt{\lambda_n(V)}$, we get
\[
\lambda_n\le \left(\sqrt{\lambda_n(V)}+R\right)^2.
\]
This concludes the proof.
\end{proof}
\begin{rem}
\label{rem:Rinfinito}
The previous estimate is quite useful for $R$ close to $0$. On the other hand, for $R$ very large, the estimate may be quite inaccurate, even in the case when the parameter $\delta$ defined in \eqref{delta} is large, as well. This depends on the geometry of the potential well $\Sigma$: for example, in the case (here $R>\delta$)
\[
\Sigma=\overline{B_{R}(0)}\setminus B_\delta(0)\qquad \text{for which}\qquad V(x)=\min\big\{(|x|-R)^2_+,(|x|-\delta)_+^2\big\},
\]
we have 
\[
\lim_{(R-\delta)\nearrow +\infty}\lambda_n(V)=0,\qquad \text{for every}\ n\in\mathbb{N}^*.
\]
This can be easily seen, by observing that\footnote{We intend that functions in $H^1_0(B_{R}(0)\setminus\overline{B_\delta(0)})$ are extended by zero, outside of the open set.} $H^1_0(B_{R}(0)\setminus\overline{B_\delta(0)})\subseteq H^1(\mathbb{R}^N;V)$ and that
\[
\mathcal{Q}_V[\phi,\phi]=\int_{B_{R}(0)} |\nabla\phi|^2\,dx,\qquad \text{for every}\ \phi\in H^1_0(B_{R}(0)\setminus\overline{B_\delta(0)}),
\]
since the potential $V$ identically vanishes on $B_{R}(0)\setminus\overline{B_\delta(0)}$. Thus, we obtain
\[
\begin{split}
\lambda_n(V)\le \lambda_{n,\rm Dir}\left(B_{R}(0)\setminus\overline{B_\delta(0)}\right)&\le \lambda_{n,\rm Dir}\left(B_\frac{R-\delta}{2}(0)\right)\\
&=\left(\frac{2}{R-\delta}\right)^2\,\lambda_{n,\rm Dir}(B_1(0)),\qquad \text{for every}\ n\in\mathbb{N}^*.
\end{split}
\]
Here, by $\lambda_{n,\rm Dir}$ we intend the $n-$th eigenvalue of the Dirichlet-Laplacian of an open set and we used that the spherical shell $B_{R}(0)\setminus\overline{B_\delta(0)}$ contains a ball with radius $(R-\delta)/2$.
\par
Observe that the previous example does not work, in the case the potential well $\Sigma$ becomes ``very large'' and at the same time ``very thin'', i.e. $R$ goes to $+\infty$ and the difference $R-\delta$ stays bounded. Indeed, in this case it may happen that the spectrum does not trivialize: we give an example in Appendix \ref{sec:app}.
\end{rem}
We also record the following continuity estimate, in terms of the distance between the potential wells. In what follows we denote by
\[
\mathrm{dist}_H(\Sigma_1,\Sigma_2)=\max\left\{\max_{x\in \Sigma_1} \mathrm{dist}(x,\Sigma_2),\,\max_{y\in \Sigma_2} \mathrm{dist}(y,\Sigma_1)\right\},
\]
the Hausdorff distance of the two compact sets $\Sigma_1,\Sigma_2\subseteq\mathbb{R}^N$.
\begin{prop}
\label{prop:hausdorff}
Let $\Sigma_1,\Sigma_2\subseteq\mathbb{R}^N$ be two non-empty compact sets. We consider the corresponding potentials
\[
V_i(x)=\left(\mathrm{dist}(x,\Sigma_i)\right)^2,\qquad \text{for}\ x\in\mathbb{R}^N.
\]
Then, for every $n\in\mathbb{N}^*$ we have 
\[
\left|\lambda_n(V_1)-\lambda_n(V_2)\right|\le \mathrm{dist}_H(\Sigma_1,\Sigma_2)\,\max\Big\{4+\frac{\lambda_n}{2},\, 2\,\lambda_n\Big\},
\]
where $\lambda_n$ is the $n-$th eigenvalue of the quantum harmonic oscillator.
\end{prop}
\begin{proof}
We first observe that if $\mathrm{dist}_H(\Sigma_1,\Sigma_2)>1$, the estimate trivially holds true. Indeed, by Lemma \ref{lm:stimarozza} we have
\[
\left|\lambda_n(V_1)-\lambda_n(V_2)\right|\le (\lambda_n(V_1)+\lambda_n(V_2))\le 2\,\lambda_n\le 2\,\lambda_n\,\mathrm{dist}_H(\Sigma_1,\Sigma_2).
\]
We can thus assume that $\mathrm{dist}_H(\Sigma_1,\Sigma_2)\le 1$. 
As in the previous proof, it is sufficient to estimate the relevant quadratic forms. We observe that 
\[
|V_1-V_2|=\left|\sqrt{V_1}-\sqrt{V_2}\right|\,\left(\sqrt{V_1}+\sqrt{V_2}\right)\le \mathrm{dist}_H(\Sigma_1,\Sigma_2)\,\left(\sqrt{V_1}+\sqrt{V_2}\right),
\]
thanks to the triangle inequality. 
Thus, for any $\phi
\in H^1(\mathbb{R}^N;V_1)$ we have that
\begin{align}
\label{uppHVnorm_gen}
\nonumber
\int_{\mathbb{R}^N} |\nabla \phi|^2\,dx
+\int_{\mathbb{R}^N} V_1\,|\phi|^2\, dx
&\leq 
\left(
\int_{\mathbb{R}^N} |\nabla \phi|^2\,dx
+\int_{\mathbb{R}^N} V_2\,|\phi|^2\, dx
\right)\\
&+\mathrm{dist}_H(\Sigma_1,\Sigma_2)\,\left(\int_{\mathbb{R}^N}\sqrt{V_1}\,|\phi|^2\,dx+\int_{\mathbb{R}^N}\sqrt{V_2}\,|\phi|^2\,dx\right).
\end{align}
Observe that if $\phi$ has unit $L^2$ norm, we have by H\"older's and Young's inequalities
\[
\begin{split}
\mathrm{dist}_H(\Sigma_1,\Sigma_2)\,\int_{\mathbb{R}^N}\sqrt{V_1}\,|\phi|^2\,dx&\le \mathrm{dist}_H(\Sigma_1,\Sigma_2)\,\left(\int_{\mathbb{R}^N}V_1\,|\phi|^2\,dx\right)^\frac{1}{2}\\
&\le \frac{\varepsilon}{2}\,\int_{\mathbb{R}^N}V_1\,|\phi|^2\,dx+\frac{1}{2\,\varepsilon}\,\left(\mathrm{dist}_H(\Sigma_1,\Sigma_2)\right)^2.
\end{split}
\]
Similarly, we get for $0<\varepsilon<1$
\[
\mathrm{dist}_H(\Sigma_1,\Sigma_2)\,\int_{\mathbb{R}^N}\sqrt{V_2}\,|\phi|^2\,dx\le \frac{\varepsilon\,(1-\varepsilon)}{2}\,\int_{\mathbb{R}^N}V_2\,|\phi|^2\,dx+\frac{1}{2\,\varepsilon\,(1-\varepsilon)}\,\left(\mathrm{dist}_H(\Sigma_1,\Sigma_2)\right)^2.
\]
By inserting these estimates in \eqref{uppHVnorm_gen}, with simple algebraic manipulations we get
\[
\begin{split}
\left(1-\frac{\varepsilon}{2}\right)\, \left(\int_{\mathbb{R}^N} |\nabla \phi|^2\,dx
+\int_{\mathbb{R}^N} V_1\,|\phi|^2\, dx\right)&\le \left(1+\frac{\varepsilon\,(1-\varepsilon)}{2}\right)\, \left(\int_{\mathbb{R}^N} |\nabla \phi|^2\,dx
+\int_{\mathbb{R}^N} V_2\,|\phi|^2\, dx\right)\\
&+\frac{2-\varepsilon}{2\,\varepsilon\,(1-\varepsilon)}\,\left(\mathrm{dist}_H(\Sigma_1,\Sigma_2)\right)^2.
\end{split}
\]
Thanks to this and the fact that $H^1(\mathbb{R}^N;V_1)=H^1(\mathbb{R}^N;V_2)$ (see again Lemma \ref{lm:hilbert}), we then get
\[
\begin{split}
\lambda_n(V_1)&\le \frac{2+\varepsilon-\varepsilon^2}{2-\varepsilon}\,\lambda_n(V_2)+\frac{1}{\varepsilon\,(1-\varepsilon)}\,\left(\mathrm{dist}_H(\Sigma_1,\Sigma_2)\right)^2\\
&=(1+\varepsilon)\,\lambda_n(V_2)+\frac{1}{\varepsilon\,(1-\varepsilon)}\,\left(\mathrm{dist}_H(\Sigma_1,\Sigma_2)\right)^2.
\end{split}
\]
By choosing 
\[
\varepsilon=\frac{1}{2}\,\mathrm{dist}_H(\Sigma_1,\Sigma_2),
\] 
we get
\[
\lambda_n(V_1)-\lambda_n(V_2)\le \frac{\lambda_n(V_2)}{2}\,\mathrm{dist}_H(\Sigma_1,\Sigma_2)+\frac{4}{2-\mathrm{dist}_H(\Sigma_1,\Sigma_2)}\,\mathrm{dist}_H(\Sigma_1,\Sigma_2).
\]
The eigenvalue $\lambda_n(V_2)$ can be bounded from above thanks to Lemma \ref{lm:stimarozza}, so to get
\[
\lambda_n(V_1)-\lambda_n(V_2)\le \mathrm{dist}_H(\Sigma_1,\Sigma_2)\,\left(4+\frac{\lambda_n}{2}\right).
\]
By exchanging the role for $V_1$ and $V_2$, we conclude.
\end{proof}

\section{Summability of eigenstates}
\label{sec:5}

In this section we prove some integrability properties of the 
eigenstates $\Psi_n$. 
Of course, Lemma~\ref{lm:intermediate} applies in particular to any $\phi\in H^1(\mathbb{R}^N;V)$. However, we will show that this integrability information can be considerably improved. We give at first a classical global $L^\infty$ estimate for the eigenstates of our operator. 
\begin{prop}
\label{prop:moser1}
We have $\Psi_n\in L^\infty(\mathbb{R}^N)$ for every $n\in\mathbb{N}^*$. More precisely, we have the following estimate 
\[
\|\Psi_n\|_{L^\infty(\mathbb{R}^N)}\le \mathrm{M}_N\,\Big(\lambda_n(V)\Big)^\frac{N}{4},
\]
where $\mathrm{M}_N>0$ is an explicit constant depending on the dimension only. 
\end{prop}
\begin{proof}
We divide the proof in three cases, according to the value of the dimension $N$. For $N\ge 2$, 
we will use a standard iterative technique {\it \`a la} Moser.
\vskip.2cm\noindent
{\it Case $N\ge 3$}. For every $\beta\ge 1$ and $M>0$, we define the non-negative continuous function
\[
f_{\beta,M}(t)=\left\{\begin{array}{ll}
\beta\,M^{\beta-1},& \mbox{ if } |t|> M,\\
\beta\,|t|^{\beta-1},& \mbox{ if } |t|\le M,\\
\end{array}
\right.
\]
and take the function
\[
F_{\beta,M}(t)=\int_0^t f_{\beta,M}(\tau)\,d\tau.
\]
Observe that 
\begin{equation}
\label{crescitaF}
0\le t\,F_{\beta,M}(t)\le |t|^{\beta+1},\qquad \mbox{ for every } t\in
\mathbb{R}.
\end{equation}
Moreover, by construction, this is a $C^1$ function with 
bounded derivative. In particular, we still have 
\[
F_{\beta,M}(\Psi_n)\in H^1(\mathbb{R}^N;V).
\]
We can thus use this function as a feasible test function in the weak 
formulation of the eigenvalue equation. We get
\[
\int_{\mathbb{R}^N} f_{\beta,M}(\Psi_n)\,|\nabla \Psi_n|^2\,dx+
\int_{\mathbb{R}^N} V\, \Psi_n\, F_{\beta,M}(\Psi_n)\,dx=\lambda_n(V)\,
\int_{\mathbb{R}^N} \Psi_n\,F_{\beta,M}(\Psi_n)\,dx.
\]
In particular, since $V$ is non-negative, we get
\[
\int_{\mathbb{R}^N} f_{\beta,M}(\Psi_n)\,|\nabla \Psi_n|^2\,dx\le \lambda_n(V)\,
\int_{\mathbb{R}^N} \Psi_n\,F_{\beta,M}(\Psi_n)\,dx.
\]
We now introduce the function
\[
G_{\beta,M}(t)=\int_0^t \sqrt{f_{\beta,M}(\tau)}\,d\tau,
\]
then the last estimate can be rewritten as
\begin{equation}
\label{basemoser}
\int_{\mathbb{R}^N} |\nabla (G_{\beta,M} \circ \Psi_n)|^2\,dx\le \lambda_n(V)\,\int_{\mathbb{R}^N} \Psi_n\,F_{\beta,M}(\Psi_n)\,dx.
\end{equation}
Observe that by construction $G_{\beta,M}$ is again a $C^1$ function, with bounded derivative. Thus $(G_{\beta,M} \circ \Psi_n)\in H^1(\mathbb{R}^N)$ and we can apply the Sobolev inequality \eqref{sobolev}, so to obtain
\begin{equation}
\label{vrooom!!}
\mathcal{T}_N\,\left(\int_{\mathbb{R}^N} |G_{\beta,M}(\Psi_n)|^{2^*}\,dx\right)^\frac{2}{2^*}\le \lambda_n(V)\,\int_{\mathbb{R}^N} \Psi_n\,F_{\beta,M}(\Psi_n)\,dx.
\end{equation}
In order to get that $\Psi_n\in L^q(\mathbb{R}^N)$ for every $2\le q<+\infty$, we define the sequence of exponents $\{\vartheta_i\}_{i\in\mathbb{N}}$ as follows
\[
\vartheta_0=1,\qquad \vartheta_{i+1}=\frac{2^*}{2}\,\vartheta_i=\left(\frac{N}{N-2}\right)^{i+1}.
\]
We then prove the following implication
\begin{equation}
\label{iterativo}
\Psi_n\in L^{2\,\vartheta_i}(\mathbb{R}^N)\qquad \Longrightarrow \qquad \Psi_n\in L^{2\,\vartheta_{i+1}}(\mathbb{R}^N).
\end{equation}
Indeed, let us suppose that $\Psi_n\in L^{2\,\vartheta_i}(\mathbb{R}^N)$.
We use \eqref{vrooom!!} with $\beta=2\,\vartheta_i-1$, so to obtain
\[
\begin{split}
\mathcal{T}_N\,\left(\int_{\mathbb{R}^N} |G_{\beta,M}(\Psi_n)|^{2^*}\,dx\right)^\frac{2}{2^*}&\le \lambda_n(V)\,\int_{\mathbb{R}^N} \Psi_n\,F_{\beta,M}(\Psi_n)\,dx\le \lambda_n(V)\,\int_{\mathbb{R}^N} \left|\Psi_n\right|^{2\vartheta_i}\,dx.
\end{split}
\]
We also used \eqref{crescitaF}, in the second inequality. We can now pass to the limit as $M$ goes to $+\infty$ on the left-hand side: by observing that
\[
|G_{\beta,M}(t)|\le\frac{2\,\sqrt{\beta}}{\beta+1}\,|t|^\frac{1+\beta}{2}\qquad \text{and}\qquad \lim_{M\to+\infty}|G_{\beta,M}(t)|=\frac{2\,\sqrt{\beta}}{\beta+1}\,|t|^\frac{1+\beta}{2},\ \mbox{ for every } t\in\mathbb{R},
\]
an application of Fatou's Lemma leads to
\[
\frac{2\,\vartheta_i-1}{\vartheta_i^2}\,\mathcal{T}_N\,\left(\int_{\mathbb{R}^N} |\Psi_n|^{2^*\vartheta_i}\,dx\right)^\frac{2}{2^*}
\le \lambda_n(V)\,\int_{\mathbb{R}^N} |\Psi_n|^{2\vartheta_i}\,dx.
\]
By using that $2^*\vartheta_i=2\,\vartheta_{i+1}$ and the fact that $\Psi_n\in L^{2\,\vartheta_i}(\mathbb{R}^N)$ we get $\Psi_n\in L^{2\,\vartheta_{i+1}}(\mathbb{R}^N)$, as desired.
\par
By observing that $\vartheta_i$ is an increasing sequence diverging to $+\infty$ and that $\Psi_n\in L^{2\,\vartheta_0}(\mathbb{R}^N)$ by assumption (observe that $2\,\vartheta_0=2$), we can iterate \eqref{iterativo} as many times as we wish, so to get
\[
\Psi_n\in L^q(\mathbb{R}^N),\qquad \text{for every}\ 2\le q<+\infty.
\]
We are now ready to prove the claimed $L^\infty$ bound. We keep the same notation as in the previous part. We have seen that 
\begin{equation}
\label{tassellino}
\left(\int_{\mathbb{R}^N} |\Psi_n|^{2\,\vartheta_{i+1}}\,dx\right)^\frac{2}{2^*}
\le \frac{\lambda_n(V)}{\mathcal{T}_N}\,\vartheta_i\,\int_{\mathbb{R}^N} |\Psi_n|^{2\,\vartheta_i}\,dx,
\end{equation}
for every $i\in\mathbb{N}$. Observe that we used
\[
\frac{2\,\vartheta_i-1}{\vartheta_i^2}\ge \frac{1}{\vartheta_i}.
\]
We take the power $1/(2\,\vartheta_i)$ on both sides of \eqref{tassellino}, so to get
\begin{equation}
\label{moser}
\left(\int_{\mathbb{R}^N} |\Psi_n|^{2\,\vartheta_{i+1}}\,dx\right)^\frac{1}{2\,\vartheta_{i+1}}
\le \left(\frac{\lambda_n(V)}{\mathcal{T}_N}\right)^\frac{1}{2\,\vartheta_i}\,\vartheta_i^\frac{1}{2\,\vartheta_i}\,\left(\int_{\mathbb{R}^N} |\Psi_n|^{2\,\vartheta_i}\,dx\right)^\frac{1}{2\,\vartheta_i}.
\end{equation}
We now iterate $k$ times \eqref{moser}, by starting from $i=0$: this gives
\begin{equation}
\label{iterato}
\left(\int_{\mathbb{R}^N} |\Psi_n|^{2\,\vartheta_{k+1}}\,dx\right)^\frac{1}{2\,\vartheta_{k+1}}\le \left(\frac{\lambda_n(V)}{\mathcal{T}_N}\right)^{\sum\limits_{i=0}^k\frac{1}{2\vartheta_i}}\,\prod_{i=0}^k \vartheta_i^\frac{1}{2\,\vartheta_i}\,\left(\int_{\mathbb{R}^N} |\Psi_n|^{2}\,dx\right)^\frac{1}{2}.
\end{equation} 
Observe that 
\[
\lim_{k\to\infty}\sum_{i=0}^k\frac{1}{2\vartheta_i}=\frac{1}{2}\,\sum_{i=0}^\infty\left(\frac{N-2}{N}\right)^i=\frac{1}{2}\,\frac{1}{1-\dfrac{N-2}{N}}=\frac{N}{4},
\]
and that 
\[
\lim_{k\to\infty} \prod_{i=0}^k \vartheta_i^\frac{1}{2\,\vartheta_i}=:C_N<+\infty.
\]
This shows that we can take the limit as $k$ goes to $\infty$ in \eqref{iterato}. By further observing that 
\[
\lim_{k\to\infty}\left(\int_{\mathbb{R}^N} |\Psi_n|^{2\,\vartheta_{k+1}}\,dx\right)^\frac{1}{2\,\vartheta_{k+1}}=\|\Psi_n\|_{L^\infty(\mathbb{R}^N)},
\]
we obtain
\[
\|\Psi_n\|_{L^\infty(\mathbb{R}^N)}\le C\,\left(\frac{\lambda_n(V)}{\mathcal{T}_N}\right)^\frac{N}{4}\,\|\Psi_n\|_{L^2(\mathbb{R}^N)}.
\]
By recalling that $\Psi_n$ has unit $L^2$ norm, we conclude the proof in the case $N\ge 3$.
\vskip.2cm\noindent
{\it Case $N=2$}. We already know that $\Psi_n\in L^q(\mathbb{R}^2)$, for every $2\le q<+\infty$, thanks to \eqref{lady}. By keeping the same notation as before, we go back to \eqref{basemoser} and multiply both sides by 
\[
\int_{\mathbb{R}^2} |G_{\beta,M}(\Psi_n)|^2\,dx.
\]
This yields
\[
\begin{split}
\left(\int_{\mathbb{R}^2} |\nabla (G_{\beta,M} \circ \Psi_n)|^2\,dx\right)&\,\left(\int_{\mathbb{R}^2} |G_{\beta,M}(\Psi_n)|^2\,dx\right)\\
&\le \lambda_n(V)\,\left(\int_{\mathbb{R}^2} |\Psi_n|^{\beta+1}\,dx\right)\,\left(\int_{\mathbb{R}^2} |G_{\beta,M}(\Psi_n)|^2\,dx\right),
\end{split}
\]
where we also used \eqref{crescitaF}.
We apply \eqref{lady} with $q=4$ on the left hand-side, so to get
\[
\pi\,\int_{\mathbb{R}^2} |G_{\beta,M}(\Psi_n)|^4\,dx\le \lambda_n(V)\,\left(\int_{\mathbb{R}^2} |\Psi_n|^{\beta+1}\,dx\right)\,\left(\int_{\mathbb{R}^2} |G_{\beta,M}(\Psi_n)|^2\,dx\right).
\]
We now take the limit as $M$ goes to $+\infty$ and get
\[
\pi\,\left(\frac{2\,\sqrt{\beta}}{\beta+1}\right)^2\,\int_{\mathbb{R}^2} |\Psi_n|^{4\frac{\beta+1}{2}}\,dx\le \lambda_n(V)\,\left(\int_{\mathbb{R}^2} |\Psi_n|^{\beta+1}\,dx\right)^2.
\]
This time, we define the sequence of exponents $\{\vartheta_i\}_{i\in\mathbb{N}}$ as follows
\[
\vartheta_0=1,\qquad \vartheta_{i+1}=2\,\vartheta_i=2^{i+1},
\]
and use the above estimate with $\beta=2\,\vartheta_i-1$, again. This gives
\[
\pi\,\frac{(2\,\vartheta_i-1)}{\vartheta_i^2}\,\int_{\mathbb{R}^2} |\Psi_n|^{2\vartheta_{i+1}}\,dx\le \lambda_n(V)\,\left(\int_{\mathbb{R}^2} |\Psi_n|^{2\vartheta_i}\,dx\right)^2.
\]
We can now proceed as in the case $N\ge 3$ and get the desired conclusion.
\vskip.2cm\noindent
{\it Case $N=1$}. This is the simplest case, the required result follows immediately from \eqref{morrey}.
\end{proof}
\begin{rem}
\label{rem:uniformLinfty}
By using the estimate of Lemma \ref{lm:stimarozza}, we can also infer the following uniform estimate
\[
\|\Psi_n\|_{L^\infty(\mathbb{R}^N)}\le \mathrm{M}_N\,\lambda_n^\frac{N}{4},\qquad \text{for every}\ n\in\mathbb{N}^*,
\]
where $\lambda_n$ is the $n-$th eigenvalue of the quantum harmonic oscillator. Observe that the upper bound does not depend on the potential well $\Sigma$.
\end{rem}
We now record a weighted integrability result on eigenstates. Apart for being interesting in itself, it will be useful in the sequel.
\begin{prop}
\label{prop:perognik}
For every $n\in\mathbb{N}^*$ and $k\in\mathbb{N}$, we have 
\[
\int_{\mathbb{R}^N} |\nabla \Psi_n|^2\,|x|^{2\,k}\,dx+\int_{\mathbb{R}^N} |\Psi_n|^2\,|x|^{2\,k+2}\,dx\le C_{N,n,k,R}.
\]
\end{prop}
\begin{proof}
We observe at first that the statement is true for $k=0$, i.e. we have
\[
\int_{\mathbb{R}^N} |\nabla\Psi_n|^2\,dx+\int_{\mathbb{R}^N} |\Psi_n|^2\,|x|^{2}\,dx\le C_{N,n,R}.
\]
Indeed, we have
\[
\int_{\mathbb{R}^N} |\Psi_n|^2\,|x|^2\,dx<+\infty,
\]
directly from Lemma \ref{lm:V}. More precisely, from \eqref{V1} with $\sigma=2$ we have 
\[
\begin{split}
\int_{\mathbb{R}^N} |x|^2\,|\Psi_n|^2\,dx&\le 4\,R^2\,\int_{B_{2 R}(0)} |\Psi_n|^2\,dx+
4\,
\int_{\mathbb{R}^N\setminus B_{2 R}(0)}V\,|\Psi_n|^2\,dx\\
&\le 4\,R^2+\,4\,\lambda_n(V).
\end{split}
\]
On account of Lemma \ref{lm:stimarozza}, this gives
\begin{equation}
\label{momentosecondo}
\int_{\mathbb{R}^N} |\Psi_n|^2\,|x|^2\,dx\le 4\,R^2+\,4\,\lambda_n.
\end{equation}
In order to conclude, we will prove the following recursive gain of weighted integrability
\begin{equation}
\label{schemapesi}
\begin{array}{c}
\mbox{\it if we have }\quad \displaystyle\int_{\mathbb{R}^N} |\nabla\Psi_n|^2\,|x|^{2\,k}\,dx+\int_{\mathbb{R}^N} |\Psi_n|^2\,|x|^{2\,k+2}\,dx\le C_{N,n,k,R},\\
\mbox{\it then we have }\quad \displaystyle\int_{\mathbb{R}^N} |\nabla\Psi_n|^2\,|x|^{2\,k+2}\,dx+\int_{\mathbb{R}^N} |\Psi_n|^2\,|x|^{2\,k+4}\,dx\le C_{N,n,k+1,R}\quad \mbox{\it as well}.
\end{array}
\end{equation}
This will be sufficient to get the claimed result. Let us suppose that for a $k\in\mathbb{N}$ we have
\[
\int_{\mathbb{R}^N} |\nabla\Psi_n|^2\,|x|^{2\,k}\,dx+\int_{\mathbb{R}^N} |\Psi_n|^2\,|x|^{2\,k+2}\,dx\le C_{N,n,k,R}.
\]
Let $M>0$ and let $\eta_M$ be a Lipschitz cut-off function such that
\[
0\le \eta_M\le 1,\quad \eta_M\equiv 1 \mbox{ on }B_M(0),\quad \eta_M\equiv 0 \mbox{ on } \mathbb{R}^N\setminus B_{M+1}(0),
\]
and
\[
|\nabla \eta_M|\le 1.
\]
We then take the test function
\[
\varphi=\Psi_n\,|x|^{2\,k+2}\,\eta_M^2,
\]
in the weak formulation. Observe that this is feasible, thanks to the properties of both $\Psi_n$ and $\eta_M$. We get
\[
\begin{split}
\int_{\mathbb{R}^N} |\nabla\Psi_n|^2\,|x|^{2\,k+2}\,\eta_M^2\,dx&+\int_{\mathbb{R}^N} V\,|\Psi_n|^2\,|x|^{2\,k+2}\,\eta_M^2\,dx\\&=\lambda_n(V)\,\int_{\mathbb{R}^N} |\Psi_n|^2\,|x|^{2\,k+2}\,\eta_M^2\,dx\\
&-(2\,k+2)\,\int_{\mathbb{R}^N} \langle \nabla\Psi_n,x\rangle\,\Psi_n\,|x|^{2\,k}\,\eta_M^2\,dx\\
&-2\int_{\mathbb{R}^N} \langle \nabla\Psi_n,\nabla \eta_M\rangle\,\Psi_n\,|x|^{2\,k+2}\,\eta_M\,dx.
\end{split}
\]
By using Young's inequality on the last two integrals, for every $M>0$ and $\delta>0$ we get
\[
\begin{split}
\int_{\mathbb{R}^N} |\nabla\Psi_n|^2\,|x|^{2\,k+2}\,\eta_M^2\,dx+\int_{\mathbb{R}^N} V\,|\Psi_n|^2\,|x|^{2\,k+2}\,\eta_M^2\,dx&\le \lambda_n(V)\,\int_{\mathbb{R}^N} |\Psi_n|^2\,|x|^{2\,k+2}\,\eta_M^2\,dx\\
&+\delta\,\int_{\mathbb{R}^N} |\nabla\Psi_n|^2\,|x|^{2\,k+2}\,\eta_M^2\,dx\\
&+\frac{(k+1)^2}{\delta}\,\int_{\mathbb{R}^N} |\Psi_n|^2\,|x|^{2\,k}\,\eta_M^2\,dx\\
&+\delta\,\int_{\mathbb{R}^N} |\nabla\Psi_n|^2\,|x|^{2\,k+2}\,\eta_M^2\,dx\\
&+\frac{1}{\delta}\,\int_{\mathbb{R}^N} |\nabla\eta_M|^2\,|x|^{2\,k+2}\,|\Psi_n|^2\,dx.
\end{split}
\]
We can take $\delta=1/4$ so to absorb the two integrals in the right-hand side containing $\nabla\Psi_n$. This gives
\[
\begin{split}
\frac{1}{2}\,\int_{\mathbb{R}^N} |\nabla\Psi_n|^2\,|x|^{2\,k+2}\,\eta_M^2\,dx+\int_{\mathbb{R}^N} V\,|\Psi_n|^2\,|x|^{2\,k+2}\,\eta_M^2\,dx&\le \lambda_n(V)\,\int_{\mathbb{R}^N} |\Psi_n|^2\,|x|^{2\,k+2}\,\eta_M^2\,dx\\
&+4\,(k+1)^2\,\int_{\mathbb{R}^N} |\Psi_n|^2\,|x|^{2\,k}\,\eta_M^2\,dx\\
&+4\,\int_{\mathbb{R}^N} |\nabla\eta_M|^2\,|x|^{2\,k+2}\,|\Psi_n|^2\,dx.
\end{split}
\]
Thanks to the properties of $\eta_M$, this implies that 
\begin{equation}
\label{letturattenta}
\begin{split}
\frac{1}{2}\,\int_{B_M(0)} |\nabla\Psi_n|^2\,|x|^{2\,k+2}\,dx&+\int_{B_M(0)} V\,|\Psi_n|^2\,|x|^{2\,k+2}\,dx\\
&\le \left(\lambda_n+4\right)\,\int_{\mathbb{R}^N} |\Psi_n|^2\,|x|^{2\,k+2}\,dx\\
&+4\,(k+1)^2\,\int_{\mathbb{R}^N} |\Psi_n|^2\,|x|^{2\,k}\,dx.
\end{split}
\end{equation}
We used again Lemma \ref{lm:stimarozza}, to bound the eigenvalue.
We observe that by H\"older's inequality and the fact that $\Psi_n$ has unit $L^2$ norm, we obtain
\[
\int_{\mathbb{R}^N} |\Psi_n|^2\,|x|^{2\,k}\,dx\le \left(\int_{\mathbb{R}^N} |\Psi_n|^2\,|x|^{2\,k+2}\,dx\right)^\frac{k}{k+1}\le \left(C_{N,n,k,R}\right)^\frac{k}{k+1}.
\]
By letting $M$ go to $+\infty$ in \eqref{letturattenta} and using Lemma \ref{lm:V} for $\phi=\Psi_n\,|x|^{k+1}$, we prove \eqref{schemapesi}. As already explained, this concludes the proof. 
\end{proof}
As a consequence of the previous result, we also get the following
\begin{coro}
\label{coro:perognik}
For every $n\in\mathbb{N}^*$ and every $2\le p<\infty$ we have 
\[
V\,\Psi_n\in L^p(\mathbb{R}^N).
\]
More precisely, there exists a constant $C=C(N,n,R,p)>0$ such that
\[
\int_{\mathbb{R}^N} V^p\,|\Psi_n|^p\,dx\le C,\qquad \text{for every}\ R\ge 0.
\]
\end{coro}
\begin{proof}
For every $2\le p<+\infty$, we have 
\[
\int_{\mathbb{R}^N} V^p\,|\Psi_n|^p\,dx\le \|\Psi_n\|_{L^\infty(\mathbb{R}^N)}^{p-2}\,\int_{\mathbb{R}^N} |V|^p\,|\Psi_n|^2\,dx.
\]
By using \eqref{uppEstVSigma} and Remark \ref{rem:uniformLinfty}, we obtain 
\[
\int_{\mathbb{R}^N} V^p\,|\Psi_n|^p\,dx\le \left(\mathrm{M}_N\,\lambda_n^\frac{N}{4}\right)^{p-2}\,\int_{\mathbb{R}^N} (|x|+R)^{2\,p}\,|\Psi_n|^2\,dx.
\]
The last term can be bounded from above, thanks to Proposition \ref{prop:perognik}.
\end{proof}

\section{Higher regularity of eigenstates}
\label{sec:6}

The following simple result will be needed in order to guarantee that a certain test function will be admissible.
\begin{lm}
\label{lm:traslata}
For every $\mathbf{h}\in\mathbb{R}^N$ we have 
\[
L^2_{\rm loc}(\mathbb{R}^N)\cap L^2(\mathbb{R}^N;V)=L^2_{\rm loc}(\mathbb{R}^N)\cap L^2(\mathbb{R}^N;\tau_\mathbf{h} V),
\]
and 
\[
H^1(\mathbb{R}^N;V)=H^1(\mathbb{R}^N;\tau_\mathbf{h}V).
\]
\end{lm}
\begin{proof}
It is sufficient to observe that 
\[
\tau_\mathbf{h} V(x)=\Big(\mathrm{dist}(x+\mathbf{h},\Sigma)\Big)^2=\Big(\mathrm{dist}(x,\Sigma-\mathbf{h})\Big)^2.
\]
Thus the potential $\tau_\mathbf{h}V$ still belongs to the same class under consideration.
In light of Lemma \ref{lm:V}, this is enough to get the first equality. The second one follows by using the same observation and Lemma \ref{lm:hilbert}.
\end{proof}
Our eigenstates belong to a higher order Sobolev space. This is the content of the following slightly more general result.
\begin{theorem}
\label{thm:domain}
For $f\in L^2(\mathbb{R}^N)$, let $\Phi\in H^1(\mathbb{R}^N;V)$ be the weak solution of
\[
\mathcal{H}_V[\Phi]=f.
\]
Then we have 
\[
\Phi\in H^2(\mathbb{R}^N) \cap \Big\{\phi\in L^2(\mathbb{R}^N)\, :\, V\,\phi\in L^2(\mathbb{R}^N)\Big\}.
\]
In particular, the domain of $\mathcal{H}_V$ is given by
\[
\mathfrak{D}(\mathcal{H}_V)= H^2(\mathbb{R}^N) \cap \Big\{\phi\in L^2(\mathbb{R}^N)\, :\, V\,\phi\in L^2(\mathbb{R}^N)\Big\}.
\]
\end{theorem}
\begin{proof}
We know that $\Phi\in H^1(\mathbb{R}^N;V)\subseteq L^2(\mathbb{R}^N)$, thanks to Lemma \ref{lm:hilbert}. We first need to prove that 
\[
\int_{\mathbb{R}^N} V^2\,|\Phi|^2\,dx<+\infty.
\]
In light of Lemma \ref{lm:negative}, it is sufficient to prove that 
\begin{equation}
\label{spocchia}
\int_{\mathbb{R}^N} |x|^4\,|\Phi|^2\,dx<+\infty.
\end{equation}
This can be proved by repeating almost verbatim the argument of the proof of Proposition \ref{prop:perognik}: we only need to treat more carefully the right-hand side. As before, let $M>0$ and let $\eta_M$ be a Lipschitz cut-off function such that
\[
0\le \eta_M\le 1,\quad \eta_M\equiv 1 \mbox{ on }B_M(0),\quad \eta_M\equiv 0 \mbox{ on } \mathbb{R}^N\setminus B_{M+1}(0),
\]
and
\[
|\nabla \eta_M|\le 1.
\]
Upon testing the equation with $\varphi=\Phi\,|x|^2\,\eta_M^2$, we get
\[
\begin{split}
\int_{\mathbb{R}^N} |\nabla\Phi|^2\,|x|^2\,\eta_M^2\,dx+\int_{\mathbb{R}^N} V\,|\Phi|^2\,|x|^2\,\eta_M^2\,dx&=\int_{\mathbb{R}^N} f\,\Phi\,|x|^2\,\eta_M^2\,dx\\
&-2\,\int_{\mathbb{R}^N} \langle \nabla\Phi,x\rangle\,\Phi\,\eta_M^2\,dx\\
&-2\,\int_{\mathbb{R}^N} \langle \nabla\Phi,\nabla \eta_M\rangle\,\Phi\,|x|^2\,\eta_M\,dx.
\end{split}
\]
By using Young's inequality on the right-hand side and \eqref{Vepsilon} with $\varepsilon=1$ on the left-hand side, for every $M>0$ and $\delta>0$ we get
\[
\begin{split}
\int_{\mathbb{R}^N} |\nabla\Phi|^2\,|x|^2\,\eta_M^2\,dx+\frac{1}{2}\int_{\mathbb{R}^N} |\Phi|^2\,|x|^4\,\eta_M^2\,dx&\le R^2\,\int_{\mathbb{R}^N}|x|^2\, |\Phi|^2\,\eta_M^2\,dx+
\frac{1}{2\,\delta}\,\int_{\mathbb{R}^N} |f|^2\,\eta_M^2\,dx\\
&+\frac{\delta}{2}\,\int_{\mathbb{R}^N} |\Phi|^2\,|x|^4\,\eta_M^2\,dx\\
&+\delta\,\int_{\mathbb{R}^N} |\nabla\Phi|^2\,|x|^2\,\eta_M^2\,dx+\frac{1}{\delta}\,\int_{\mathbb{R}^N} |\Phi|^2\,\eta_M^2\,dx\\
&+\delta\,\int_{\mathbb{R}^N} |\nabla\Phi|^2\,|x|^2\,\eta_M^2\,dx\\
&+\frac{1}{\delta}\,\int_{\mathbb{R}^N} |\nabla\eta_M|^2\,|x|^2\,|\Phi|^2\,dx.
\end{split}
\]
If we now take $\delta=1/4$, we can absorb the two integrals in the right-hand side containing $\nabla\Phi$, as well as the term $|\Phi|^2\,|x|^4$. This gives
\[
\begin{split}
\frac{1}{2}\,\int_{\mathbb{R}^N} |\nabla\Phi|^2\,|x|^2\,\eta_M^2\,dx+\frac{3}{8}\,\int_{\mathbb{R}^N} |\Phi|^2\,|x|^4\,\eta_M^2\,dx&\le 
R^2\,\int_{\mathbb{R}^N}|x|^2\, |\Phi|^2\,\eta_M^2\,dx+
2\,\int_{\mathbb{R}^N} |f|^2\,\eta_M^2\,dx\\
&+4\,\int_{\mathbb{R}^N} |\Phi|^2\,\eta_M^2\,dx\\
&+4\,\int_{\mathbb{R}^N} |\nabla\eta_M|^2\,|x|^2\,|\Phi|^2\,dx.
\end{split}
\]
Thanks to the properties of $\eta_M$, this implies that 
\[
\begin{split}
\frac{1}{2}\,\int_{B_M(0)} |\nabla\Phi|^2\,|x|^2\,dx+\frac{3}{8}\,\int_{B_M(0)} |\Phi|^2\,|x|^4\,dx&\le 
(4+R^2)\,\int_{\mathbb{R}^N}|x|^2\, |\Phi|^2\,dx+
2\,\int_{\mathbb{R}^N} |f|^2\,dx\\
&+4\,\int_{\mathbb{R}^N} |\Phi|^2\,dx.
\end{split}
\]
By letting $M$ go to $+\infty$ and recalling Lemma \ref{lm:V}, we get \eqref{spocchia}.
\par
In order to prove the $H^2$ regularity, we use the classical Nirenberg-Stampacchia method, based on differentiating in a discrete sense the equation. We set 
\[
F=f-V\,\Phi,
\]
thus the function $\Phi$ weakly solves the Poisson equation
\[
-\Delta \Phi=F\in L^2(\mathbb{R}^N),
\]
i.e.
\[
\int_{\mathbb{R}^N} \langle \nabla \Phi,\nabla\varphi\rangle\,dx=\int_{\mathbb{R}^N} F\,\varphi\,dx,\qquad \text{for every}\ \varphi\in H^1(\mathbb{R}^N;V).
\]
By a simple change of variable, for every $\mathbf{h}\in\mathbb{R}^N\setminus\{0\}$ we have that $\tau_\mathbf{h}\Phi$ satisfies
\[
\int_{\mathbb{R}^N} \langle \nabla (\tau_\mathbf{h}\Phi),\nabla\varphi\rangle\,dx=\int_{\mathbb{R}^N} \tau_\mathbf{h}F\,\varphi\,dx,\qquad \text{for every}\ \varphi\in H^1(\mathbb{R}^N;V).
\] 
Observe that we also used Lemma \ref{lm:traslata}, to ensure that $H^1(\mathbb{R}^N;\tau_\mathbf{h}V)=H^1(\mathbb{R}^N;V).$
We subtract from this equation the one satisfied by $\Phi$ and then take the test function $\varphi=\delta_\mathbf{h} \Phi$. 
This yields
\begin{equation}
\label{differential}
\int_{\mathbb{R}^N}|\delta_\mathbf{h} \nabla \Phi|^2\,dx=\int_{\mathbb{R}^N} \delta_\mathbf{h}F\,\delta_{\mathbf{h}}\Phi\,dx.
\end{equation}
We take $\mathbf{h}=h\,\mathbf{e}_k$ for $h\in\mathbb{R}\setminus\{0\}$ and $k\in\{1,\dots,N\}$, then we use the following semi-discrete integration by parts formula 
\[
\int_{\mathbb{R}^N} \delta_{h\,\mathbf{e}_k} F\,\delta_{h\,\mathbf{e}_k} \Phi\,dx=-h\,\int_{\mathbb{R}^N} \frac{\partial}{\partial x_k}\delta_{h\,\mathbf{e}_k} \Phi\,\left(\int_0^1 F(x+t\,h\,\mathbf{e}_k)\,dt\right)\,dx,
\]
see for example \cite[Lemma 4.5]{Lin}.
Let us set for simplicity 
\[
W(x)=\int_0^1 F(x+t\,h\,\mathbf{e}_k)\,dt,
\]
by using that $F\in L^2(\mathbb{R}^N)$ we get that $W\in L^2(\mathbb{R}^N)$, as well. Moreover, by Jensen's inequality and the translation invariance of $L^p$ norms, we have
\[
\begin{split}
\|W\|^2_{L^2(\mathbb{R}^N)}&=\int_{\mathbb{R}^N}\left|\int_0^1 F(x+t\,h\,\mathbf{e}_k)\,dt\right|^2\,dx\\
&\le \int_{\mathbb{R}^N}\int_0^1 |F(x+t\,h\,\mathbf{e}_k)|^2\,dt\,dx=\int_0^1\left(\int_{\mathbb{R}^N} |F(x+t\,h\,\mathbf{e}_k)|^2\,dx\right)\,dt=\|F\|_{L^2(\mathbb{R}^N)}^2.
\end{split}
\]
From equation \eqref{differential} with $\mathbf{h}=h\,\mathbf{e}_k$, we thus get
\[
\int_{\mathbb{R}^N}|\delta_{h\,\mathbf{e}_k} \nabla \Phi|^2\,dx=-h\,\int_{\mathbb{R}^N} W\,\delta_{h\,\mathbf{e}_k}\left(\frac{\partial \Phi}{\partial x_k}\right)\,dx.
\]
By applying H\"older's inequality, this in turn gives
\[
\int_{\mathbb{R}^N}|\delta_{h\,\mathbf{e}_k} \nabla \Phi|^2\,dx\le |h|\,\|W\|_{L^2(\mathbb{R}^N)}\,\left\|\delta_{h\,\mathbf{e}_k}\frac{\partial \Phi}{\partial x_k}\right\|_{L^2(\mathbb{R}^N)}.
\]
The Hessian term on the right-hand side can now be absorbed in the left-hand side: indeed, by Young's inequality, we have 
\begin{equation}
\label{differential2}
\int_{\mathbb{R}^N}|\delta_{h\,\mathbf{e}_k} \nabla \Phi|^2\,dx\le\frac{|h|^2}{2}\,\|W\|^2_{L^2(\mathbb{R}^N)}+\frac{1}{2}\,\left\|\delta_{h\,\mathbf{e}_k}\frac{\partial \Phi}{\partial x_k}\right\|_{L^2(\mathbb{R}^N)}^2,
\end{equation}
and then we notice that
\[
\begin{split}
\left\|\delta_{h\,\mathbf{e}_k}\frac{\partial \Phi}{\partial x_k}\right\|_{L^2(\mathbb{R}^N)}^2&=\int_{\mathbb{R}^N} \left|\frac{\partial \Phi}{\partial x_k}(x+h\,\mathbf{e}_k)-\frac{\partial \Phi}{\partial x_k}(x)\right|^2\,dx\\
&\le \int_{\mathbb{R}^N} \left|\nabla \Phi(x+h\,\mathbf{e}_k)-\nabla \Phi(x)\right|^2\,dx=\int_{\mathbb{R}^N}|\delta_{h\,\mathbf{e}_k} \nabla \Phi|^2\,dx.
\end{split}
\]
Thus, from \eqref{differential2} we get
\[
\int_{\mathbb{R}^N}|\delta_{h\,\mathbf{e}_k} \nabla \Phi|^2\,dx\le |h|^2\,\|W\|^2_{L^2(\mathbb{R}^N)}\le |h|^2\,\|F\|^2_{L^2(\mathbb{R}^N)}=|h|^2\,\|f-V\,\Phi\|_{L^2(\mathbb{R}^N)}.
\]
By dividing everything by $|h|^2$ and appealing to the characterization of Sobolev spaces in terms of finite differences (see for example \cite[Chapter 8]{Gi}), we get the desired regularity for $\Phi$.
\vskip.2cm\noindent
As for the domain of $\mathcal{H}_V$, the first part of the proof and Lemma \ref{lm:hilbert} give
\[
\mathfrak{D}(\mathcal{H}_V)=\Big\{\phi\in H^1(\mathbb{R}^N;V)\, :\, -\Delta \phi+V\, \phi\in L^2(\mathbb{R}^N)\Big\}\subseteq H^2(\mathbb{R}^N) \cap \Big\{\phi\in L^2(\mathbb{R}^N)\, :\, V\,\phi\in L^2(\mathbb{R}^N)\Big\}.
\]
In order to prove the reverse inclusion, it is sufficient to prove that
\[
\Big\{\phi\in L^2(\mathbb{R}^N)\, :\, V\,\phi\in L^2(\mathbb{R}^N)\Big\}\subseteq L^2(\mathbb{R}^N;V).
\]
This easily follows from H\"older's inequality
\[
\int_{\mathbb{R}^N} V\,|\phi|^2\,dx\le\left(\int_{\mathbb{R}^N}|\phi|^2\,dx\right)^\frac{1}{2}\,\left(\int_{\mathbb{R}^N}V^2\,|\phi|^2\,dx\right)^\frac{1}{2},
\]
thus concluding the proof.
\end{proof}
As for classical regularity of eigenstates, we have the following result.
\begin{prop}
\label{prop:classical}
We have $\Psi_n\in C^{2,\alpha}_{\rm loc}(\mathbb{R}^N)$ for every $0<\alpha<1$. In particular, $\Psi_n$ solves the eigenvalue equation
\[
-\Delta \Psi_n+V\,\Psi_n=\lambda_n(V)\,\Psi_n,\qquad \text{in}\ \mathbb{R}^N,
\]
in classical sense. 
\end{prop}
\begin{proof}
Let us set
\[
F=\lambda_n(V)\,\Psi_n-V\,\Psi_n.
\]
Thanks to Proposition \ref{prop:moser1} and Corollary \ref{coro:perognik}, we have $F\in L^q(\mathbb{R}^N)$ for every $2\le q< +\infty$. Thus we have
\[
-\Delta \Psi_n=F\in L^q(\mathbb{R}^N),\qquad \text{for every}\ 2\le q<+\infty.
\]
By the classical {\it Calder\'on-Zygmund estimates} (see for example \cite[Theorem 9.9]{GT}), this implies that $\Psi_n\in W^{2,p}_{\rm loc}(\mathbb{R}^N)$, for every $1<p<+\infty$. By the {\it Sobolev Embedding Theorem}, this in turn implies that $\Psi_n\in C^{1,\alpha}_{\rm loc}(\mathbb{R}^N)$, for every $0<\alpha<1$.
By using this regularity gain and recalling that $V$ is locally Lipschitz, we can then infer that
\[
-\Delta \Psi_n=F\in C^{0,\alpha}_{\rm loc}(\mathbb{R}^N),
\]
for every $0<\alpha<1$. An application of {\it Schauder's estimates} (see for example \cite[Corollary 6.3]{GT}) now gives the claimed regularity.
\end{proof}
\begin{rem}[Maximal regularity]
\label{oss:maxreg}
In general, we can not expect the eigenstates $\Psi_n$ to belong to $C^3$. Indeed, let us consider the positive ground state $\Psi_1$ and let us suppose that the potential $V$ is not $C^1$ at the origin: this happens for example when the potential well $\Sigma$ is given by the torus \eqref{TorusRr}. 
\par
We argue by contradiction and assume that $\Psi_1\in C^3_{\rm loc}(\mathbb{R}^N)$. In particular, the Laplacian $\Delta \Psi_1$ would be differentiable at $x=0$. Since $\Psi_1$ is a classical solution of the equation, we get the pointwise identity
\[
V=\lambda_1(V)+\frac{\Delta \Psi_1}{\Psi_1},
\]
where we also used that $\Psi_1>0$, by the minimum principle, as already observed. The previous identity would imply that $V$ is differentiable at the origin, while this is not the case. We thus obtain a contradiction.
\end{rem}

\section{Exponential decay at infinity}
\label{sec:7}

We will show exponential decay for the eigenstates, in the sup norm. This is quite a classical result (see for example \cite{Ag, HS, Sim}). Our proof is elementary and relies only on the weak form of the equation, in conjunction with suitable test function arguments. The result will follow by applying iterative arguments {\it \`a la} De Giorgi and Moser, thus the proof is genuinely nonlinear in nature.
\par
As a preliminary result, we start with an exponential decay in the $L^2$ norm. We pay due attention to the constants involved in the estimate.
\begin{lm}[Exponential decay in $L^2$]
\label{lm:expoL2}
There exist an exponent $\alpha=\alpha(N)>1$ and a constant
$C_1=C_1(N,n,R)>0$ such that
\[
\int_{\mathbb{R}^N\setminus B_\varrho(0)}|\Psi_n|^2\,dx
\le C_1\,e^{-\varrho\,\log\alpha}, \qquad \mbox{ for every } \varrho\ge 0.
\]
\end{lm}
\begin{proof}
We first observe that it is sufficient to prove that there exist $\alpha=\alpha(N)>1$, $R_0=R_0(N,n,R)>0$ and 
$C=C(N,n,R)>0$ such that
\begin{equation}
\label{avanti!}
\int_{\mathbb{R}^N\setminus B_\varrho(0)}
|\Psi_n|^2\,dx\le C\,e^{-\varrho\,\log\alpha}, \qquad \mbox{for every}\ 
\varrho\ge R_0.
\end{equation}
Indeed, for $0\le \varrho<R_0$ we would trivially have
\[
\int_{\mathbb{R}^N\setminus B_\varrho(0)}
|\Psi_n|^2\,dx\le 1\le \frac{e^{-\varrho\,\log \alpha}}{e^{-R_0\,\log \alpha}},
\]
and thus the claimed estimate would follow by taking
\[
C_1=\max\left\{C,e^{R_0\,\log \alpha}\right\}.
\]
For every $\varrho>0$, we take the Lipschitz cuf-off function 
$\eta_\varrho$ such that
\[
\eta_\varrho \equiv 0 \mbox{ in } B_\varrho(0),\qquad \eta_
\varrho\equiv 1 \mbox{ in } \mathbb{R}^N\setminus B_{\varrho+1}(0),
\qquad 0\le \eta_\varrho\le 1,
\]
and
\[
|\nabla \eta_\varrho|=1\qquad \text{on}\ B_{\varrho+1}(0)
\setminus B_\varrho(0).
\]
We then insert in the weak formulation the test function 
$\varphi=\Psi_n\,\eta^2_\varrho$. We get
\begin{align*}
\int_{\mathbb{R}^N} |\nabla \Psi_n|^2\,\eta_\varrho^2
+\int_{\mathbb{R}^N} V\,|\Psi_n|^2\,\eta_\varrho^2\,dx
&=\lambda_n(V) \,\int_{\mathbb{R}^N} |\Psi_n|^2\,\eta_\varrho^2
\,dx-2\,\int_{\mathbb{R}^N} 
\langle\nabla \Psi_n ,\nabla \eta_\varrho\rangle
\,\eta_\varrho\,\Psi_n\,dx.
\end{align*}
On the last integral, we apply the Cauchy-Schwarz and Young 
inequalities, so to get
\[
-2\,\int_{\mathbb{R}^N} \langle\nabla\Psi_n, 
\nabla \eta_\varrho\rangle\, \eta_\varrho\,\Psi_n\,dx\le 
\int_{\mathbb{R}^N} |\nabla \Psi_n|^2\,\eta_\varrho^2
\,dx+\int_{\mathbb{R}^N} |\nabla\eta_\varrho|^2\,|\Psi_n|^2\,dx.
\]
By inserting this estimate in the identity above and canceling out the 
common factor, we get
\[
\int_{\mathbb{R}^N} V\,|\Psi_n|^2\,\eta_\varrho^2
\,dx\le \lambda_n(V)\,\int_{\mathbb{R}^N} |\Psi_n|^2\,\eta_\varrho^2\,dx
+\int_{\mathbb{R}^N} |\nabla\eta_\varrho|^2\,|\Psi_n|^2\,dx.
\]
In particular, by using the properties of $\eta_\varrho$, 
this entails that
\begin{equation}
\label{schemadecayL2}
\int_{\mathbb{R}^N\setminus B_{\varrho+1}(0)} V\,|\Psi_n|^2
\,dx\le \lambda_n(V)\,\int_{\mathbb{R}^N\setminus B_\varrho(0)} 
|\Psi_n|^2\,dx+\int_{B_{\varrho+1}(0)\setminus 
B_\varrho(0)} |\Psi_n|^2\,dx.
\end{equation}
On account of the confining property \eqref{confining}, we have
that there exists $R_0>0$ such that 
\[
\inf_{\mathbb{R}^N\setminus B_{\varrho+1}(0)} V\ge 2\,\lambda_n, \qquad \mbox{ for every } \varrho\ge R_0.
\]
More precisely, by recalling \eqref{lowEstVSigma}, we easily see that we can take 
\begin{equation}
\label{R0}
R_0=R_0(N,n,R):=\sqrt{2\,\lambda_n}+R.
\end{equation}
Thus, for every $\varrho\ge R_0$, from \eqref{schemadecayL2} and Lemma \ref{lm:stimarozza} we get
\[
\begin{split}
2\,\lambda_n\,\int_{\mathbb{R}^N\setminus B_{\varrho+1}(0)
}|\Psi_n|^2\,dx&\le \lambda_n\,\int_{\mathbb{R}^N\setminus B_
\varrho(0)} |\Psi_n|^2\,dx
+\int_{B_{\varrho+1}(0)\setminus B_\varrho(0)} |\Psi_n|^2\,dx.
\end{split}
\]
By decomposing 
\[
\int_{\mathbb{R}^N\setminus B_\varrho(0)} |\Psi_n|^2\,dx=
\int_{\mathbb{R}^N\setminus B_{\varrho+1}(0)} |\Psi_n|^2\,dx+
\int_{B_{\varrho+1}(0)\setminus B_\varrho(0)} |\Psi_n|^2\,dx,
\]
this can be recast into
\begin{equation}
\label{schemadecayL2b}
\int_{\mathbb{R}^N\setminus B_{\varrho+1}(0)} |\Psi_n|^2\,dx
\le \left(\frac{1}{\lambda_n}+1\right)\,\int_{B_{\varrho+1}(0)\setminus 
B_\varrho(0)} |\Psi_n|^2\,dx.
\end{equation}
We set for brevity
\[
m(\varrho)=\int_{\mathbb{R}^N\setminus B_{\varrho}(0)} 
|\Psi_n|^2\,dx\qquad \text{and}\qquad \Theta=\frac{1}{\lambda_n}+1.
\]
Accordingly, the estimate \eqref{schemadecayL2b} can be rewritten as
\[
m(\varrho+1)\le \Theta\,
\Big(m(\varrho)-m(\varrho+1)\Big),\qquad 
\text{for every}\ \varrho\ge R_0.
\]
In turn, we rewrite this as follows
\begin{equation}
\label{schemadecayL2c}
m(\varrho+1)\le \frac{\Theta}{1+\Theta}
\,m(\varrho),\qquad \mbox{ for every }\varrho\ge R_0.
\end{equation}
By iterating this estimate, we get the claimed exponential 
decay of $m(\varrho)$. 
Indeed, it is sufficient to observe that $\varrho\mapsto m(\varrho)$ 
is non-increasing and thus for every $\varrho\ge R_0+1$ we have
\begin{align*}
m(\varrho)\le m(R_0+\big\lfloor \varrho-R_0\big\rfloor)&\le 
\left(
\frac{\Theta}{1+\Theta}
\right)^{\big\lfloor \varrho-R_0\big\rfloor}\,m(R_0)
\le \left(\frac{\Theta}{1+\Theta}
\right)^{\varrho-R_0-1}\,m(R_0).
\end{align*}
Here, for every $t\in\mathbb{R}$, we denoted its {\it integer part} by 
\[
\big\lfloor t\big\rfloor=\max\Big\{n\in\mathbb{Z}\, :\, t\ge n\Big\}.
\]
This has been obtained by applying a suitable number of times 
\eqref{schemadecayL2c}. By recalling that $m(R_0)\le m(0)=1$, 
we then get
\[
\int_{\mathbb{R}^N\setminus B_{\varrho}(0)} |\Psi_n|^2\,dx
\le \left(
\frac{\Theta}{1+\Theta}
\right)^{\varrho}\,
\left(\frac{1+\Theta}{\Theta}
\right)^{R_0+1},\qquad \mbox{ for every } \varrho \ge R_0+1.
\]
We now observe that by definition
\[
\frac{1}{2}\le \frac{\Theta}{1+\Theta}\le \frac{\lambda_1+1}{2\,\lambda_1+1}=:\frac{1}{\alpha}.
\]
Observe that $\alpha$ depends on the dimension $N$ only, through the first eigenvalue $\lambda_1$ of the quantum harmonic oscillator.
Thus, we obtain in particular
\[
\int_{\mathbb{R}^N\setminus B_{\varrho}(0)} |\Psi_n|^2\,dx
\le \alpha^{-\varrho}\,
2^{R_0+1},\qquad \mbox{ for every } \varrho \ge R_0+1.
\]
By choosing
\[
C=2^{R_0+1},
\]
we get \eqref{avanti!}, as desired.
\end{proof}
\begin{rem}
Observe that the costant $C_1$ is given by
\[
C_1=\max\left\{2^{R_0+1},e^{R_0 \log \alpha}\right\}=2^{R_0+1},
\]
with $R_0$ defined by \eqref{R0}. In particular, we see that $C_1\nearrow +\infty $ as
\[
R\nearrow +\infty \qquad\text{or}\qquad n\to\infty,
\]
while $C_1$ stays uniformly bounded, as $R$ goes to $0$.
\end{rem}
We can now prove the pointwise exponential decay: this is the main result of this section.
\begin{theorem}[Exponential decay in $L^\infty$]
\label{teo:expoLinfty}
There exists a constant $C_2=C_2(N,n,R)>0$
such that
\[
0\le |\Psi_n(x)|\le C_2\,e^{-\frac{|x|}{2}\,\log \alpha},\qquad \text{for every}\ x\in
\mathbb{R}^N,
\]
where $\alpha=\alpha(N)>1$ is the same exponent as in Lemma \ref{lm:expoL2}.
\end{theorem}
\begin{proof}
It is enough to prove a $L^\infty-L^2$ estimate, localized ``at infinity'', i.e. an estimate like
\begin{equation}
\label{stimachemiserve}
\|\Psi_n\|_{L^\infty(\mathbb{R}^N\setminus B_{\varrho+1}(0))}\le C\,\|\Psi_n\|_{L^2(\mathbb{R}^N\setminus B_\varrho(0))},\qquad \text{for every}\ \varrho\ge 0,
\end{equation}
for $C=C(N,n)>0$.
Then, by recalling that $\Psi_n\in L^\infty(\mathbb{R}^N)$, we would eventually get the conclusion by joining this estimate and Lemma \ref{lm:expoL2}. 
\par
Indeed, if $|x|\le 1$, we would get from Remark \ref{rem:uniformLinfty}
\[
|\Psi_n(x)|\le \mathrm{M}_N\,\lambda_n^\frac{N}{4}\le \frac{\mathrm{M}_N\,\lambda_n^\frac{N}{4}}{e^{-\frac{\log \alpha}{2}}}\,e^{-\frac{|x|}{2}\,\log \alpha}.
\]
On the other hand, if $|x|>1$, then $\varrho+1\le |x|\le \varrho+2$ for some $\varrho\ge 0$. Accordingly, we would get
\[
\begin{split}
|\Psi_n(x)|\le \|\Psi_n\|_{L^\infty(\mathbb{R}^N\setminus B_{\varrho+1}(0))}&\le C\,\|\Psi_n\|_{L^2(\mathbb{R}^N\setminus B_\varrho(0))}\\
&\le C\,\sqrt{C_1}\,e^{-\varrho\,\frac{\log\alpha}{2}}\\
&\le C\,\sqrt{C_1}\,\frac{e^{-\varrho\,\frac{\log\alpha}{2}}}{e^{-(\varrho+2)\,\frac{\log\alpha}{2}} }\,e^{-|x|\,\frac{\log\alpha}{2}}=C\,\alpha\,\sqrt{C_1}\,e^{-|x|\,\frac{\log\alpha}{2}},
\end{split}
\]
where $C_1$ is the same constant as in Lemma \ref{lm:expoL2}. In conclusion, we would get the claimed estimate with constant
\[
C_2=\max\left\{\sqrt{\alpha}\,\mathrm{M}_N\,\lambda_n^\frac{N}{4},\,C\,\alpha\,\sqrt{C_1}\right\}.
\]
The estimate \eqref{stimachemiserve} can be obtained by appealing once again to a suitable Moser--type iteration. For every pair of radii $0<r<R$ and every exponent $\beta\ge 1$, we insert in the weak formulation the test function
\[
\varphi=|\Psi_n|^{\beta-1}\,\Psi_n\,\eta^2_{r,R},
\]
where $\eta_{r,R}$ is a Lipschitz cut-off function such that
\[
\eta_{r,R} \equiv 0\ \mbox{in}\ B_{r}(0),\qquad \eta_{r,R}\equiv 1\ \mbox{in}\ \mathbb{R}^N\setminus B_{R}(0),\qquad 0\le \eta_R\le 1, 
\]
and
\[
|\nabla \eta_{r,R}|=\frac{1}{R-r},\qquad \text{in}\ B_R(0)\setminus B_r(0).
\]
This is a feasible test function, thanks to the Chain Rule in Sobolev spaces (recall that $\Psi_n\in L^\infty(\mathbb{R}^N)$ by Proposition \ref{prop:moser1}). We then get
\[
\begin{split}
\frac{4\,\beta}{(\beta+1)^2}\,\int_{\mathbb{R}^N} \left|\nabla |\Psi_n|^\frac{\beta+1}{2}\right|^2\,\eta_{r,R}^2\,dx&+\int_{\mathbb{R}^N} V\,|\Psi_n|^{\beta+1}\,\eta_{r,R}^2\,dx\\
&=\lambda_n(V)\,\int_{\mathbb{R}^N} |\Psi_n|^{\beta+1}\,\eta_{r,R}^2\,dx\\
&-2\,\int_{\mathbb{R}^N} \langle \nabla \Psi_n,\nabla \eta_{r,R}\rangle\,\eta_{r,R}\,|\Psi_n|^{\beta-1}\,\Psi_n\,dx.
\end{split}
\]
On the last term, we use as usual the Cauchy-Schwarz and Young inequalities, so to get
\[
\begin{split}
-2\,\int_{\mathbb{R}^N} \langle \nabla \Psi_n,\nabla \eta_{r,R}\rangle\,\eta_{r,R}\,|\Psi_n|^{\beta-1}\,\Psi_n\,dx&\le \delta\,\int_{\mathbb{R}^N} |\nabla \Psi_n|^2\,|\Psi_n|^{\beta-1}\,\eta_{r,R}^2\,dx\\
&+\frac{1}{\delta}\,\int_{\mathbb{R}^N} |\nabla \eta_{r,R}|^2\,|\Psi_n|^{\beta+1}\,dx\\
&=\delta\,\frac{4}{(\beta+1)^2}\,\int_{\mathbb{R}^N} \left|\nabla |\Psi_n|^\frac{\beta+1}{2}\right|^2\,\eta_{r,R}^2\,dx\\
&+\frac{1}{\delta}\,\int_{\mathbb{R}^N} |\nabla \eta_{r,R}|^2\,|\Psi_n|^{\beta+1}\,dx,
\end{split}
\]
which holds for every $\delta>0$. In particular, by taking $\delta=\beta/2$ we can absorb the gradient term on the right-hand side and obtain
\[
\begin{split}
\frac{2\,\beta}{(\beta+1)^2}\,\int_{\mathbb{R}^N} \left|\nabla |\Psi_n|^\frac{\beta+1}{2}\right|^2\,\eta_{r,R}^2\,dx+\int_{\mathbb{R}^N} V\,|\Psi_n|^{\beta+1}\,\eta_{r,R}^2\,dx&\le\lambda_n(V)\,\int_{\mathbb{R}^N} |\Psi_n|^{\beta+1}\,\eta_{r,R}^2\,dx\\
&+\frac{2}{\beta}\,\int_{\mathbb{R}^N} |\nabla \eta_{r,R}|^2\,|\Psi_n|^{\beta+1}\,dx.
\end{split}
\]
We now set for brevity $\vartheta=(\beta+1)/2$, drop the (non-negative) term containing $V$ and use some elementary manipulations, so to get
\[
\int_{\mathbb{R}^N} \left|\nabla |\Psi_n|^\vartheta\right|^2\,\eta_{r,R}^2\,dx\le 2\,\vartheta\,\lambda_n(V)\,\int_{\mathbb{R}^N} |\Psi_n|^{2\,\vartheta}\,\eta_{r,R}^2\,dx+4\,\int_{\mathbb{R}^N} |\nabla \eta_{r,R}|^2\,|\Psi_n|^{2\,\vartheta}\,dx.
\]
We add on both sides the term
\[
\int_{\mathbb{R}^N} |\nabla \eta_{r,R}|^2\,|\Psi_n|^{2\,\vartheta}\,dx,
\]
and use that
\[
\int_{\mathbb{R}^N} |\nabla \eta_{r,R}|^2\,|\Psi_n|^{2\,\vartheta}\,dx+\int_{\mathbb{R}^N} \left|\nabla |\Psi_n|^\vartheta\right|^2\,\eta_{r,R}^2\,dx\ge \frac{1}{2}\,\int_{\mathbb{R}^N} \left|\nabla \left(|\Psi_n|^\vartheta\,\eta_{r,R}\right)\right|^2\,dx.
\]
This gives
\[
\int_{\mathbb{R}^N} \left|\nabla \left(|\Psi_n|^\vartheta\,\eta_{r,R}\right)\right|^2\,dx\le 4\,\vartheta\,\lambda_n(V)\,\int_{\mathbb{R}^N} |\Psi_n|^{2\,\vartheta}\,\eta_{r,R}^2\,dx+10\,\int_{\mathbb{R}^N} |\nabla \eta_{r,R}|^2\,|\Psi_n|^{2\,\vartheta}\,dx.
\]
For simplicity, we now confine ourselves to the case $N\ge 3$: the cases $N=2$ and $N=1$ can be treated with minor modifications, as in the proof of Proposition \ref{prop:moser1}.
\par
Thus, we can bound the left-hand side from below, again thanks to the Sobolev inequality. We obtain
\[
\mathcal{T}_N\,\left(\int_{\mathbb{R}^N}\left(|\Psi_n|^\vartheta\,\eta_{r,R}\right)^{2^*}\,dx\right)^\frac{2}{2^*}\le 4\,\vartheta\,\lambda_n(V)\,\int_{\mathbb{R}^N} |\Psi_n|^{2\,\vartheta}\,\eta_{r,R}^2\,dx+10\,\int_{\mathbb{R}^N} |\nabla \eta_{r,R}|^2\,|\Psi_n|^{2\,\vartheta}\,dx.
\]
Thanks to the properties of $\eta_{r,R}$, this in turn implies that
\[
\mathcal{T}_N\,\left(\int_{\mathbb{R}^N\setminus B_{R}(0)}|\Psi_n|^{2^*\vartheta}\right)^\frac{2}{2^*}\le 4\,\vartheta\,\lambda_n(V)\,\int_{\mathbb{R}^N\setminus B_r(0)} |\Psi_n|^{2\,\vartheta}\,dx+\frac{10}{(R-r)^2}\,\int_{\mathbb{R}^N\setminus B_{r}(0)} |\Psi_n|^{2\,\vartheta}\,dx.
\]
With some manipulations more, we can also obtain
\begin{equation}
\label{premoserdecay}
\left(\int_{\mathbb{R}^N\setminus B_{R}(0)}|\Psi_n|^{2^*\vartheta}\right)^\frac{1}{2^*\vartheta}\le \left(\frac{10\,\vartheta}{\mathcal{T}_N}\,\left(\lambda_n(V)+\frac{1}{(R-r)^2}\right)\right)^\frac{1}{2\,\vartheta}\,\left(\int_{\mathbb{R}^N\setminus B_r(0)}|\Psi_n|^{2\,\vartheta}\,dx\right)^\frac{1}{2\,\vartheta}.
\end{equation}
We now introduce the sequence of exponents
\[
\vartheta_0=1,\qquad \vartheta_{i+1}=\frac{2^*}{2}\,\vartheta_i=\left(\frac{N}{N-2}\right)^{i+1},\qquad \mbox{for}\ i\in\mathbb{N}.
\]
together with the sequence of radii
\[
R_i=\varrho+1-\frac{1}{2^i},\qquad \mbox{ for } i\in\mathbb{N},
\]
where $\varrho\ge 0$ is fixed.
Observe that $R_0=\varrho$, $R_\infty=\varrho+1$ and $R_{i+1}-R_i=1/2^{i+1}$. By using \eqref{premoserdecay} with $\vartheta=\vartheta_i$, $r=R_i$ and $R=R_{i+1}$, we obtain
\[
\left(\int_{\mathbb{R}^N\setminus B_{R_{i+1}}(0)}|\Psi_n|^{2\,\vartheta_{i+1}}\right)^\frac{1}{2\,\vartheta_{i+1}}\le \left(\frac{10\,\vartheta_i}{\mathcal{T}_N}\,\left(\lambda_n+4^{i+1}\right)\right)^\frac{1}{2\,\vartheta_i}\,\left(\int_{\mathbb{R}^N\setminus B_{R_i}(0)}|\Psi_n|^{2\,\vartheta_i}\,dx\right)^\frac{1}{2\,\vartheta_i}.
\]
Observe that we also used Lemma \ref{lm:stimarozza}, in order to estimate $\lambda_n(V)$.
Starting from $i=0$ and iterating this estimate infinitely many times, we get the desired conclusion \eqref{stimachemiserve}.
\end{proof}

\section{Stability of eigenspaces}
\label{sec:8}

We already know from Lemma \ref{lm:stimarozza} that
\[
\sqrt{\lambda_n}-R\le \sqrt{\lambda_n(V)}\le \sqrt{\lambda_n},\qquad \text{for every}\ n\in\mathbb{N}^*.
\]
In particular, we get that the spectrum of our operator
collapses to the spectrum of the quantum harmonic oscillator, as $R$ goes to $0$.
\par
In this section we want to prove a similar kind of {\it quantitative} result, for the relevant eigenspaces.
\begin{theorem}
\label{thm:stability}
For $n\in\mathbb{N}^*$, we set
\[
\mathcal{W}_n=\Big\{\phi\in H^1(\mathbb{R}^N;|x|^2)\, :\, \phi\ \text{is an eigenstate relative to}\ \lambda_n\Big\}.
\]
Then there exists an explicit constant $C_3=C_3(N,n)>0$ such that
\[
\mathrm{dist}_{L^2(\mathbb{R}^N)}(\Psi_n;\mathcal{W}_n)\le C_3\, \sqrt{R},\qquad \text{for every}\ R>0.
\]
\end{theorem}
\begin{proof}
We divide the proof in various steps, for ease of readability.
\vskip.2cm\noindent
{\it Step 1: set-up}.
We need to fix at first some notation. We call
\[
\kappa(n)=\mathrm{dim\,} \mathcal{W}_n-1,
\]
then there exists an index $j_n\in\{0,\dots,n-1\}$ such that\footnote{For $n=1$, we have already observed in Remark \ref{rem:ground} that $\kappa(1)=0$. Accordingly, in this case we have $j_1=0$.}
\begin{equation}
\label{multiple}
\lambda_{n-j_n}=\dots=\lambda_n=\dots=\lambda_{n-j_n+\kappa(n)}.
\end{equation}
By construction, we thus have
\begin{equation}
\label{delta}
\delta_n:=\lambda_{n-j_n+\kappa(n)+1}-\lambda_{n}>0.
\end{equation}
We then set 
\[
R_1=\frac{\sqrt{\lambda_n}-\sqrt{\lambda_{n-j_n-1}}}{2},
\]
with the notation $\lambda_0:=0$.
According to Lemma \ref{lm:stimarozza}, we have for every $0<R\le R_1$
\begin{equation}
\label{bassolambda}
\sqrt{\lambda_n(V)}\ge \sqrt{\lambda_n}-R\ge\frac{\sqrt{\lambda_n}+\sqrt{\lambda_{n-j_n-1}}}{2}>\sqrt{\lambda_{n-j_n-1}}.
\end{equation}
Let us indicate by $\{\Phi_n\}_{n\in\mathbb{N}^*}$ an orthonormal basis of eigenstates for the quantum harmonic oscillator. For every $k\in\mathbb{N}^*$, we define  
\[
\widehat{\Psi_n}(k)=\int_{\mathbb{R}^N} \Psi_n\,\Phi_k\,dx,
\]
i.e. the {\it $k-$th Fourier coefficient }of $\Psi_n$. We then have
\begin{equation}
\label{normafourier}
1=\|\Psi_n\|^2_{L^2(\mathbb{R}^N)}=\sum_{k=1}^\infty |\widehat{\Psi_n}(k)|^2.
\end{equation}
Finally, we observe that by orthonormality
\begin{equation}
\label{distanza}
\begin{split}
\left(\mathrm{dist}_{L^2(\mathbb{R}^N)}(\Psi_n;\mathcal{W}_n)\right)^2&=\left\|\Psi_n-\sum_{k=n-j_n}^{n-j_n+\kappa(n)} \widehat{\Psi_n}(k)\, \Phi_k\right\|_{L^2(\mathbb{R}^N)}^2\\
&=\sum_{k=1}^{n-j_n-1} |\widehat{\Psi_n}(k)|^2+\sum_{k=n-j_n+\kappa(n)+1}^{\infty} |\widehat{\Psi_n}(k)|^2:=\mathcal{J}_1+\mathcal{J}_2.
\end{split}
\end{equation}
It is intended that the first term $\mathcal{J}_1$ is void in the case $j_n=n-1$.
Thus, in order to get the claimed estimate, it is sufficient to suitably estimate the last two sums. Before going further, we observe that it is sufficient to get the claimed estimate for $R\le R_1$: indeed, for every $R>R_1$ we would simply get
\[
\left(\mathrm{dist}_{L^2(\mathbb{R}^N)}(\Psi_n;\mathcal{W}_n)\right)^2=\mathcal{J}_1+\mathcal{J}_2\le \sum_{k=1}^\infty |\widehat{\Psi_n}(k)|^2=1\le \frac{1}{R_1}\,R,
\]
i.e. we have the desired estimate with $C_3=1/R_1$, the latter depending only on $N$ and $n$, by definition.
\vskip.2cm\noindent
{\it Step 2: high frequencies estimate}. We show how the estimate of $\mathcal{J}_2$ can reduced to the estimate of $\mathcal{J}_1$.
By using that $\Psi_n$ is an eigenstate corresponding to $\lambda_n(V)$, we have
\[
\begin{split}
\lambda_n(V)&=\int_{\mathbb{R}^N} |\nabla \Psi_n|^2\,dx+\int_{\mathbb{R}^N} V\,|\Psi_n|^2\,dx\\
&\ge \sum_{k=1}^\infty |\widehat{\Psi_n}(k)|^2\,\lambda_k+\frac{1}{1+\varepsilon}\,\int_{\mathbb{R}^N} |x|^2\,|\Psi_n|^2\,dx-\frac{R^2}{\varepsilon}\\
&\ge \sum_{k=1}^\infty |\widehat{\Psi_n}(k)|^2\,\lambda_k-\frac{R^2}{\varepsilon},
\end{split}
\]
for every $\varepsilon>0$. In the first inequality we used \eqref{Vepsilon}.
We estimate the series as follows
\[
\begin{split}
\sum_{k=1}^\infty |\widehat{\Psi_n}(k)|^2\,\lambda_k&=\sum_{k=n-j_n+\kappa(n)+1}^\infty |\widehat{\Psi_n}(k)|^2\,\lambda_k+\lambda_n\,\sum_{k=n-j_n}^{n-j_n+\kappa(n)}|\widehat{\Psi_n}(k)|^2+\sum_{k=1}^{n-j_n-1} |\widehat{\Psi_n}(k)|^2\,\lambda_k\\
&\ge \lambda_{n-j_n+\kappa(n)+1}\,\mathcal{J}_2+\lambda_n\,\sum_{k=n-j_n}^{n-j_n+\kappa(n)}|\widehat{\Psi_n}(k)|^2+\lambda_1\, \mathcal{J}_1.
\end{split}
\]
Observe that we used the monotonicity of eigenvalues with respect to $n$ and the multiplicity assumption \eqref{multiple}.
By further choosing $\varepsilon=R$ and subtracting $\lambda_n$, up to now we have obtained
\begin{equation}
\label{stimaeigenstate}
\begin{split}
\big(\lambda_n(V)-\lambda_n\big)&+R
\ge \lambda_{n-j_n+\kappa(n)+1}\,\mathcal{J}_2+\lambda_n\,\left(\sum_{k=n-j_n}^{n-j_n+\kappa(n)}|\widehat{\Psi_n}(k)|^2-1\right)+\lambda_1\,\mathcal{J}_1.
\end{split}
\end{equation}
Thanks to \eqref{normafourier}, we can write
\[
\sum_{k=n-j_n}^{n-j_n+\kappa(n)}|\widehat{\Psi_n}(k)|^2-1=-\mathcal{J}_1-\mathcal{J}_2.
\]
By injecting this identity in \eqref{stimaeigenstate}, we get
\[
\begin{split}
R+\big(\lambda_n(V)-\lambda_n\big)+(\lambda_n-\lambda_1)\,\mathcal{J}_1\ge (\lambda_{n-j_n+\kappa(n)+1}-\lambda_n)\,\mathcal{J}_2.
\end{split}
\]
Finally, by recalling \eqref{delta} and Lemma \ref{lm:stimarozza}, we arrive at the estimate
\begin{equation}
\label{daviderompe}
\begin{split}
\frac{R}{\delta_n}+\frac{\lambda_n-\lambda_1}{\delta_n}\,\mathcal{J}_1\ge \mathcal{J}_2,
\end{split}
\end{equation}
which concludes this step. The previous estimate holds for every $R>0$.
\vskip.2cm\noindent
{\it Step 3: low frequencies estimate}.
We now estimate the finite sum $\mathcal{J}_1$. We will derive some ``almost orthogonality'' relations. Indeed, by using first the equation for $\Phi_k$ and then that for $\Psi_n$, we get
\[
\begin{split}
\widehat{\Psi_n}(k)=\int_{\mathbb{R}^N} \Psi_n\,\Phi_k\,dx&=\frac{1}{\lambda_k}\,\left[\int_{\mathbb{R}^N} \langle \nabla \Phi_k,\nabla\Psi_n\rangle\,dx+\int_{\mathbb{R}^N} |x|^2\,\Phi_k\,\Psi_n\,dx\right]\\
&=\frac{1}{\lambda_k}\,\left[\int_{\mathbb{R}^N} \langle \nabla \Phi_k,\nabla\Psi_n\rangle\,dx+\int_{\mathbb{R}^N} V\,\Phi_k\,\Psi_n\,dx\right]\\
&+\frac{1}{\lambda_k}\,\int_{\mathbb{R}^N} \big(|x|^2-V\big)\,\Phi_k\,\Psi_n\,dx\\
&=\frac{\lambda_n(V)}{\lambda_k}\,\int_{\mathbb{R}^N} \Psi_n\,\Phi_k\,dx+\frac{1}{\lambda_k}\,\int_{\mathbb{R}^N} \big(|x|^2-V\big)\,\Phi_k\,\Psi_n\,dx\\
&=\frac{\lambda_n(V)}{\lambda_k}\,\widehat{\Psi_n}(k)+\frac{1}{\lambda_k}\,\int_{\mathbb{R}^N} \big(|x|^2-V\big)\,\Phi_k\,\Psi_n\,dx.
\end{split}
\]
We thus have the following estimate
\[
\big|\lambda_k-\lambda_n(V)\big|\,|\widehat{\Psi_n}(k)|\le \int_{\mathbb{R}^N} \big||x|^2-V\big|\,|\Phi_k|\,|\Psi_n|\,dx.
\]
By recalling \eqref{bassolambda}, for $k=1,\dots,n-j_n-1$ and $R\le R_1$ we have
\[
\big|\lambda_k-\lambda_n(V)\big|=\lambda_n(V)-\lambda_k\ge \left(\frac{\sqrt{\lambda_n}+\sqrt{\lambda_{n-j_n-1}}}{2}\right)^2-\lambda_{n-j_n-1}=:\alpha_{N,n}>0.
\]
We thus obtain for $k=1,\dots,n-j_n-1$ and $R\le R_1$
\[
|\widehat{\Psi_n}(k)|\le\frac{1}{\alpha_{N,n}}\,\left(\int_{\mathbb{R}^N} \big||x|^2-V\big|\,|\Phi_k|^2\,dx\right)^\frac{1}{2}\,\left(\int_{\mathbb{R}^N} \big||x|^2-V\big|\,|\Psi_n|^2\,dx\right)^\frac{1}{2}.
\]
We show that the last two terms can be controlled in terms of $R$. Indeed,
from \eqref{eqn1} and \eqref{Vepsilon}, we have
\[
-\frac{\varepsilon}{1+\varepsilon}\,|x|^2-\frac{R^2}{\varepsilon}\le V(x)-|x|^2 \leq \varepsilon\,|x|^2 +\left(1+\frac{1}{\varepsilon}\right)\,\delta^2,
\]
for every $\varepsilon>0$. Thus, we get in particular
\[
\left|V(x)-|x|^2\right|\le\max\left\{\frac{\varepsilon}{1+\varepsilon}\,|x|^2+\frac{R^2}{\varepsilon}, \varepsilon\,|x|^2 +\left(1+\frac{1}{\varepsilon}\right)\,\delta^2\right\}.
\]
We choose again $\varepsilon=R$, so that
\begin{equation}
\label{closepot}
\left|V(x)-|x|^2\right|\le R\,\max\left\{\frac{|x|^2}{R+1}+1,|x|^2 +(R+1)\,\left(\frac{\delta}{R}\right)^2\right\}\le R\,(|x|^2+R+1).
\end{equation}
Finally, we get for $k=1,\dots,n-j_n-1$ and $R\le R_1$
\[
|\widehat{\Psi_n}(k)|\le\frac{R}{\alpha_{N,n}}\,\left(\int_{\mathbb{R}^N} |x|^2\,|\Phi_k|^2\,dx+(R+1)\right)^\frac{1}{2}\,\left(\int_{\mathbb{R}^N} |x|^2\,|\Psi_n|^2\,dx+(R+1)\right)^\frac{1}{2}.
\]
In conclusion, we get
\begin{equation}
\label{fame!}
\mathcal{J}_1\le \frac{R^2}{\alpha_{N,n}^2}\,\left(\int_{\mathbb{R}^N} |x|^2\,|\Psi_n|^2\,dx+(R+1)\right)\,\sum_{k=1}^{n-j_n-1}\left(\int_{\mathbb{R}^N} |x|^2\,|\Phi_k|^2\,dx+(R+1)\right),
\end{equation}
for every $R\le R_1$.
\vskip.2cm\noindent
{\it Step 4: conclusion}.
By joining \eqref{distanza}, \eqref{daviderompe} and \eqref{fame!} we get for every $R\le R_1$
\[
\begin{split}
\left(\mathrm{dist}_{L^2(\mathbb{R}^N)}(\Psi_n;\mathcal{W}_n)\right)^2&=\mathcal{J}_1+\mathcal{J}_2\\
&\le \frac{R}{\delta_n}+\left(1+\frac{\lambda_n-\lambda_1}{\delta_n}\right)\,\mathcal{J}_1\\
&\le \frac{R}{\delta_n}+\frac{R^2}{\alpha_{N,n}^2}\,\left(1+\frac{\lambda_n-\lambda_1}{\delta_n}\right)\,\left(\int_{\mathbb{R}^N} |x|^2\,|\Psi_n|^2\,dx+(R_1+1)\right)\\
&\times\sum_{k=1}^{n-j_n-1}\left(\int_{\mathbb{R}^N} |x|^2\,|\Phi_k|^2\,dx+(R_1+1)\right).
\end{split}
\]
In order to conclude, we only need to observe that for $R\le R_1$
\[
\int_{\mathbb{R}^N} |x|^2\,|\Psi_n|^2\,dx\le 4\,R_1^2+4\,\lambda_n,
\]
thanks to \eqref{momentosecondo}.
Thus, we eventually get the desired estimate.
\end{proof}
With a little extra work, we can improve the metric of the previous stability estimate.
\begin{coro}
\label{coro:stability}
For $n\in\mathbb{N}^*$, we still set
\[
\mathcal{W}_n=\Big\{\phi\in H^1(\mathbb{R}^N;|x|^2)\, :\, \phi\ \text{is an eigenstate relative to}\ \lambda_n\Big\}.
\]
Then there exists an explicit constant $C_4=C_4(N,n)>0$ such that
\[
\mathrm{dist}_{H^1(\mathbb{R}^N;|x|^2)}(\Psi_n;\mathcal{W}_n)\le C_4\, \sqrt{R},\qquad \text{for every}\ R>0.
\]
\end{coro}
\begin{proof}
With the previous notation, let us set 
\[
\Phi=\sum_{k=n-j_n}^{n-j_n+\kappa(n)} \widehat{\Psi_n}(k)\,\Phi_k\in \mathcal{W}_n.
\]
We thus have 
\[
\left(\mathrm{dist}_{H^1(\mathbb{R}^N;|x|^2)}(\Psi_n;\mathcal{W}_n)\right)^2\le \int_{\mathbb{R}^N} |\nabla \Psi_n-\nabla \Phi|^2\,dx+\int_{\mathbb{R}^N} |x|^2\,|\Psi_n-\Phi|^2\,dx.
\]
As above, on account of the uniform estimates at our disposal, it is sufficient to bound the last two terms for $R\le R_1$. Here, the radius $R_1=R_1(N,n)$ is the same as in the previous proof.
\par
By using the equations for both $\Phi$ and $\Psi_n$, we get
\[
\int_{\mathbb{R}^N} \langle \nabla \Phi,\nabla(\Phi-\Psi_n)\rangle\,dx+\int_{\mathbb{R}^N} |x|^2\,\Phi\,(\Phi-\Psi_n)\,dx=\lambda_n\,\int_{\mathbb{R}^N} \Phi\,(\Phi-\Psi_n)\,dx,
\]
and
\[
\int_{\mathbb{R}^N} \langle \nabla \Psi_n,\nabla(\Phi-\Psi_n)\rangle\,dx+\int_{\mathbb{R}^N} V\,\Psi_n\,(\Phi-\Psi_n)\,dx=\lambda_n(V)\,\int_{\mathbb{R}^N} \Psi_n\,(\Phi-\Psi_n)\,dx.
\]
By subtracting them, we get
\[
\begin{split}
\int_{\mathbb{R}^N} |\nabla \Psi_n-\nabla \Phi|^2\,dx+\int_{\mathbb{R}^N} |x|^2\,|\Psi_n-\Phi|^2\,dx&\le \int_{\mathbb{R}^N} \big|V-|x|^2\big|\,|\Psi_n|\,|\Psi_n-\Phi|\,dx\\
&+\big|\lambda_n(V)-\lambda_n\big|\,\int_{\mathbb{R}^N} |\Psi_n|\,|\Psi_n-\Phi|\,dx\\
&+\lambda_n\,\int_{\mathbb{R}^N} |\Psi_n-\Phi|^2\,dx\\
&\le \left(\int_{\mathbb{R}^N} \big|V-|x|^2\big|^2\,|\Psi_n|^2\,dx\right)^\frac{1}{2}\,\|\Psi_n-\Phi\|_{L^2(\mathbb{R}^N)}\\
&+\big|\lambda_n(V)-\lambda_n\big|\,\|\Psi_n-\Phi\|_{L^2(\mathbb{R}^N)}\\
&+\lambda_n\,\int_{\mathbb{R}^N} |\Psi_n-\Phi|^2\,dx.
\end{split}
\]
The $L^2$ norm of $\Psi_n-\Phi$ can be estimated by Theorem \ref{thm:stability}, while for the difference of the eigenvalues we can apply
\[
|\lambda_n(V)-\lambda_n|=\Big(\sqrt{\lambda_n}-\sqrt{\lambda_n(V)}\Big)\,\Big(\sqrt{\lambda_n}+\sqrt{\lambda_n(V)}\Big)\le 2\,R\,\sqrt{\lambda_n},
\]
which follows from Lemma \ref{lm:stimarozza}. For every $0<R\le R_1$, these yield
\[
\begin{split}
\int_{\mathbb{R}^N} |\nabla \Psi_n-\nabla \Phi|^2\,dx+\int_{\mathbb{R}^N} |x|^2\,|\Psi_n-\Phi|^2\,dx&\le C_3\,\sqrt{R_1}\,\left(\int_{\mathbb{R}^N} \big|V-|x|^2\big|^2\,|\Psi_n|^2\,dx\right)^\frac{1}{2}\\
&+2\,\sqrt{\lambda_n}\,C_3\,\sqrt{R_1}\,R+C_3^2\,\lambda_n\,R.
\end{split}
\]
We only need to estimate the integral containing the difference of the potentials. By \eqref{closepot}, we get 
\[
\begin{split}
\int_{\mathbb{R}^N} \big|V-|x|^2\big|^2\,|\Psi_n|^2\,dx&\le R^2\,\int_{\mathbb{R}^N}(|x|^2+R+1)^2\,|\Psi_n|^2\,dx\\
&\le 2\,R^2\,\left(\int_{\mathbb{R}^N}|x|^4\,|\Psi_n|^2\,dx+(R_1+1)^2\right).
\end{split}
\]
By using Proposition \ref{prop:perognik} with $k=1$ to bound uniformly the first term on the right-hand side, we conclude. 
\end{proof}

\appendix

\section{The case of $\mathbb{S}^1$ in $\mathbb{R}^3$}
\label{sec:app}

In this final part, we want to briefly comment on the case of the $\mathbb{S}^1-$type potential well \eqref{SR}, in the case $R$ becoming larger and larger. Observe that the explicit universal lower bound \eqref{roughLB} trivializes in the limit as $R$ goes to $+\infty$. We will show that this behavior is not optimal, in this peculiar case: the ground state energy must stay uniformly bounded from below, even when $R$ goes to $+\infty$.
\par
To this aim, we are going to exploit the ideas of \cite[Chapter 16]{Maz}.
\begin{prop}
Let $\Sigma\subseteq \mathbb{R}^3$ be given by \eqref{SR}. There exists a universal constant $\beta>0$ such that 
\[
\lambda_1(V)\ge \beta,\qquad \text{for every}\ R>0.
\]
\end{prop}
\begin{proof}
In view of the lower bound from Proposition \ref{prop:poincare}, it is not restrictive to assume that $R\ge 2$.
It is sufficient to prove the following Poincar\'e inequality
\[
\beta\,\int_{\mathbb{R}^3} |\phi|^2\, dx\leq
\int_{\mathbb{R}^3} |\nabla\phi|^2\, dx+\int_{\mathbb{R}^3} V\,
|\phi|^2\, dx,\qquad \text{for every}\ \phi\in C^\infty_c(\mathbb{R}^3),
\]
with a constant $\beta>0$ not depending on $R$. We proceed as in the proof of \cite[Theorem 16.2.1]{Maz}. We first tile the whole space $\mathbb{R}^3$ with the family of cubes 
\[
Q_{\mathbf{n}}:=\mathbf{n}+\left(-\frac{1}{2},\frac{1}{2}\right)^3,\qquad \text{with}\ \mathbf{n}\in \mathbb{Z}^3.
\]
Observe that each
$Q_{\mathbf{n}}$ has edges of length $d=1$, parallel to the coordinate axes. Accordingly, for every $\phi\in C^\infty_c(\mathbb{R}^3)$ we can write
\[
\int_{\mathbb{R}^3} |\phi|^2\, dx=\sum_{\mathbf{n}\in\mathbb{Z}^3} \int_{Q_{\mathbf{n}}} |\phi|^2\, dx.
\]
On each cube $Q_{\mathbf{n}}$, we apply the following weighted Poincar\'e inequality of \cite[Lemma 16.1.1]{Maz}
\begin{equation}
\label{mazzia}
\int_{Q_{\mathbf{n}}} |\phi|^2\, dx\le \frac{C}{\lambda}\,\int_{Q_{\mathbf{n}}} |\nabla\phi|^2\, dx+\frac{C}{\displaystyle\inf_{e\in \mathscr{N}_\lambda(Q_{\mathbf{n}})} \int_{Q_{\mathbf{n}}\setminus e} V\,dx}\,\int_{Q_{\mathbf{n}}} V\,
|\phi|^2\, dx,
\end{equation}
where:
\begin{itemize}
\item $C>0$ is a universal constant;
\vskip.2cm
\item $\lambda>0$ is an arbitrary positive number;
\vskip.2cm
\item $\mathscr{N}_\lambda(Q_{\mathbf{n}})$ is the collection of all compact sets $e\subseteq \overline{Q_{\mathbf{n}}}$ such that
\[
\mathrm{cap}(e;2\,Q_{\mathbf{n}})\le \lambda,
\]
where $2\, Q_{\mathbf{n}}$ denotes the cube having the same center as $Q_{\mathbf{n}}$, dilated by a factor $2$, and 
\[
\mathrm{cap}(e;\Omega)=\inf_{\phi\in C^\infty_c(\Omega)}\left\{\int_{\Omega} |\nabla\phi|^2\,dx\, :\, \phi\ge 1 \ \text{on}\ e\right\},
\]
is the {\it capacity of a compact set $e$ relative to the open set $\Omega$} containing it.
\end{itemize}
In light of this inequality, it is sufficient to assure that we can choose $\lambda>0$ small enough such that 
\begin{equation}
\label{bestemmiefinali}
\inf_{e\in\mathscr{N}_\lambda(Q_{\mathbf{n}})}\int_{Q_{\mathbf{n}}\setminus e} V\,dx\ge c_0, 
\end{equation}
for a constant $c_0>0$, independent of $R$. By summing up \eqref{mazzia} and using the tiling property of the cubes, we will eventually get the conclusion.
\par
In order to choose $\lambda$ such that \eqref{bestemmiefinali} holds, we first observe that if $B_{\mathbf{n}}$ is the ball having the same center as $Q_{\mathbf{n}}$ and radius $\sqrt{3}$, then $2\, Q_{\mathbf{n}}\subseteq B_{\mathbf{n}}$. Thus, for every $e\in\mathscr{N}_\lambda(Q_{\mathbf{n}})$ we have
\[
\mathrm{cap}(e;B_{\mathbf{n}})\le \mathrm{cap}(e;2\,Q_{\mathbf{n}})\le \lambda.
\]
By \cite[equation (2.2.10)]{Maz} we have 
\[
\mathrm{cap}(e;B_{\mathbf{n}})\ge \left(4\,\pi\right)^\frac{2}{3}\,\sqrt[3]{3}\,\frac{|B_{\mathbf{n}}|^\frac{1}{3}\,|e|^\frac{1}{3}}{|B_{\mathbf{n}}|^\frac{1}{3}-|e|^\frac{1}{3}}.
\]
By observing that 
\[
0<|B_{\mathbf{n}}|^\frac{1}{3}-|e|^\frac{1}{3}\le |B_{\mathbf{n}}|^\frac{1}{3},
\]
we get in particular 
\[
\lambda\ge \left(4\,\pi\right)^\frac{2}{3}\,\sqrt[3]{3}\,|e|^\frac{1}{3}=:\gamma\,|e|^\frac{1}{3}.
\]
Observe that both $\lambda$ and $\gamma$ are universal constant and do not depend on anything.
We choose $\lambda$ as follows
\[
\frac{\lambda}{\gamma}= \left(\frac{1}{4}\right)^\frac{1}{3}\qquad \text{that is}\qquad \lambda=\gamma\,\left(\frac{1}{4}\right)^\frac{1}{3}.
\]
This discussion guarantees that for such a choice of $\lambda$, we have 
\begin{equation}
\label{misura_e}
|e| \le \frac{1}{4}=\frac{|Q_{\mathbf{n}}|}{4},\qquad \text{for every}\ e\in\mathscr{N}_\lambda(Q_{\mathbf{n}}).
\end{equation}
By using this property and the fact that $V(x)=(\mathrm{dist}(x,\Sigma))^2$, with $\Sigma$ given by \eqref{SR}, we can now easily get \eqref{bestemmiefinali}.
Indeed, for every $\mathbf{n}=(n_1,n_2,n_3)\in\mathbb{Z}^3$ with $|n_3|\ge 1$, we have 
\[
\mathrm{dist}(x,\Sigma)\ge \frac{1}{2},\qquad \text{for every}\ x\in Q_{\mathbf{n}},
\]
which implies that 
\[
V(x)\ge \frac{1}{4},\qquad \mbox{ for every } x\in Q_{\mathbf{n}}.
\]
Accordingly, for every such $\mathbf{n}\in\mathbb{Z}^3$ we get
\[
\inf_{e\in\mathscr{N}_\lambda(Q_{\mathbf{n}})}\int_{Q_{\mathbf{n}}\setminus e} V\,dx\ge \frac{1}{4}\,\inf_{e\in\mathscr{N}_\lambda(Q_{\mathbf{n}})}|Q_{\mathbf{n}}\setminus e|\ge \frac{3}{16},
\]
thanks to \eqref{misura_e}.
We now take $\mathbf{n}\in\mathbb{Z}^3$ such that $n_3=0$. Accordingly, it may happen that the cube $Q_{\mathbf{n}}$ intersects the potential well $\Sigma$, where $V$ vanishes. However, for every such a cube, we observe that 
\[
V(x)\ge \frac{9}{256},\qquad \mbox{ for every } x\in Q_{\mathbf{n}}\ \text{such that}\ |x_3|\ge \frac{3}{16},
\]
thanks to the fact that $\Sigma$ lies in the hyperplane $\{x\in\mathbb{R}^3\,:\, x_3=0\}$.
In particular, we get in this case
\[
\begin{split}
\inf_{e\in\mathscr{N}_\lambda(Q_{\mathbf{n}})}\int_{Q_{\mathbf{n}}\setminus e} V\,dx&\ge \inf_{e\in\mathscr{N}_\lambda(Q_{\mathbf{n}})}\int_{\{x\in Q_{\mathbf{n}}\, :\, |x_3|\ge 3/16\}\setminus e} V\,dx\\
&\ge \frac{9}{256}\,\inf_{e\in\mathscr{N}_\lambda(Q_{\mathbf{n}})} \left|\{x\in Q_{\mathbf{n}}\, :\, |x_3|\ge 3/16\}\setminus e\right|\\
&\ge \frac{9}{256}\,\left(2\left(\frac{1}{2}-\frac{3}{16}\right)-\frac{1}{4}\right)=\frac{9}{256}\cdot \frac{3}{8},
\end{split}
\]
where we used again \eqref{misura_e}.
This finally gives the claimed fact \eqref{bestemmiefinali}, with $c_0=27/2048$.
\end{proof}

%

\end{document}